\documentclass[11pt,letterpaper]{article}

\usepackage[english]{babel}
\usepackage{amsfonts,amsmath, amssymb}
\usepackage{dutchcal}
\usepackage{hyperref}
\usepackage{csquotes}

\usepackage{xcolor}

\usepackage{enumerate}
\usepackage{amsfonts, mathdots}
\usepackage{amsthm}
\usepackage{url}

\usepackage{epsf,epsfig,amsfonts,graphicx}

\usepackage{amssymb}
\usepackage{amsthm}
\usepackage{epsf,epsfig,amsfonts,graphicx}

\newtheorem{Theorem}{Theorem}

\newtheorem{Lemma}{Lemma}[section]

\newtheorem{Corollary}{Corollary}[section]

\newtheorem{Ex}{Example}[section]
\newtheorem{Remark}{Remark}[section]

\theoremstyle{remark}

\newcommand{\be}{\begin{equation}}
\newcommand{\ee}{\end{equation}}

\newcommand{\R}{\mathbb{R}}\newcommand{\Id}{\textrm{\rm Id}}

\newcommand{\gl}{\mathrm{gl}}

\newcommand{\ddd}{\mathrm{d}}

\newcommand{\Sym}{\mathsf{Sym}}

\newcommand{\pd}[2]{\frac{\partial#1}{\partial#2}}

\newcommand{\Span}{\operatorname{Span}}

\newcommand{\dd}{{\mathrm d}\,}

\newcommand{\weg}[1]{}

\usepackage[
  backend=biber,
  style=numeric,
  sorting=nyt,
  isbn=false, doi=true, url=false  
]{biblatex}
\title{ Duality of operator Frobenius algebras  and solution of  Eisenhart-St\"ackel problem in the  non-diagonal case} 
\author{Alexey V.\ Bolsinov\footnote{School of Mathematics,
 Loughborough University,
 LE11 3TU, UK  and Institute of Mathematics and Mathematical Modeling, Almaty, Kazakhstan\ \ \quad {\tt A.Bolsinov@lboro.ac.uk}},    
Andrey Yu.\  Konyaev\footnote{Faculty of Mechanics and Mathematics and Center for Fundamental and Applied Mathematics, Moscow State University, 119992, Moscow, Russia   
 \ \ \quad{\tt  maodzund@yandex.ru}}  \, and 
    Vladimir S.\ Matveev\footnote{
Institut f\"ur Mathematik, Friedrich-Schiller-Universit\"at Jena,
07737 Jena,  Germany  \ \ \quad {\tt  vladimir.matveev@uni-jena.de}} \thanks{Corresponding author}}
\date{}

\begin{document}

\maketitle

\begin{abstract}
We study Frobenius algebras of operator fields and introduce a novel notion of duality for them. We show that, under the assumption that the operator fields forming the Frobenius algebra are mutual symmetries, the operator fields in the dual Frobenius algebra are also mutual symmetries. This result allows one to construct new infinite-dimensional integrable systems of hydrodynamic type starting from a given one. As the main application, we solve the long-standing Eisenhart--St\"ackel problem for any Segre characteristic and in arbitrary dimension: namely, we describe all nondegenerate finite-dimensional integrable systems whose integrals are quadratic in momenta such that the corresponding $(1,1)$-tensors commute as operator fields. 

\vspace{1ex}
\end{abstract}

\section{Introduction}

We study $n$-dimensional operator Frobenius algebras on an $n$-dimensional manifold $\mathsf M^n$. An operator Frobenius algebra is an $n$-dimensional space  of operators on the manifold (= $(1,1)$-tensor fields) such that, at every point $x\in M$, its restriction to $T_x\mathsf M^n$ is an $n$-dimensional commutative associative Frobenius algebra with respect to a certain nondegenerate symmetric bilinear form.

Notice that operator Frobenius algebras are naturally related to the notion of an $F$-structure (introduced under the name ``weakly Frobenius structure'' in \cite{ManinHertling}), which is a pre-structure for the Dubrovin--Frobenius structure introduced in \cite{Dubrovin1, Dubrovin2}. Recall that in these references as well as in subsequent works on this widely studied topic (see, e.g.,   \cite{Arsie1, Arsie2, Hertlingbook}), one considers Frobenius algebras of vector fields on the manifold. If one chooses $n$ linearly independent vector fields $\xi_1,\dots,\xi_n$ on $\mathsf M$, then the right-multiplication operators $\mathsf R_{\xi_i}$ naturally span an operator Frobenius algebra in our sense.  Conversely, given an operator Frobenius algebra $\mathcal K=\Span_\R(K_1,\dots, K_n)$,  we can naturally define a commutative associative multiplication $\star$ on the tangent bundle by choosing a unity vector field $\xi_0$. Then any tangent vector $\eta$ can be uniquely written in the form $\eta=K_\eta \xi_0$ for some  $K_\eta \in \mathcal K$ and we can set by definition $\eta\star\zeta = K_\eta K_\zeta \xi_0$.  

The study of Frobenius and $F$-structures is motivated by their applications in the theory of integrable systems; so is our study of operator Frobenius algebras. Our paper contains two lines of applications of operator Frobenius algebras, one in the theory of integrable PDE systems of hydrodynamic type, and another in the theory of finite-dimensional integrable (Hamiltonian ODE) systems.

Recall that quasilinear systems of hydrodynamic type are evolutionary   PDE systems for the unknown vector-functions $u(x,t)= (u^1(x,t), \dots, u^n(x,t))^\top$ of the form
\begin{equation}\label{eq:K0}
u_t(x,t)= K( u(x,t)) \cdot u_x(x,t), 
\end{equation}
where ``$\cdot$'' denotes matrix multiplication of the $n\times n$ matrix $K$ whose components depend on   $u$   with the vector $u_x= \tfrac{\partial }{\partial x}u$. Note that  $K$  has the geometric nature of an operator field on a manifold with local coordinates $u^1,\dots, u^n$, since $K$ transforms according to the tensor rule for $(1,1)$-tensors under diffeomorphic changes of $u^1,\dots,u^n$. Such systems are of great interest for applications and  describe many physics-related processes, see e.g.  \cite{Lax, Serre1, Serre2,Serre3,  LeVe}. Their  study   is a prominent and classical topic, with first nontrivial results obtained already by B. Riemann in 1860 \cite{Riemann}.

 It is known that integrability of hydrodynamic type systems \eqref{eq:K0}  is closely related to the existence of the  so-called  mutual symmetries $K_1,\dots, K_n$, $K=K_1$ (the formal differential-geometric  definition will be given in Section \ref{s1}; in the language of PDEs, this means that  for any $i,j$  the operators commute algebraically and   the PDE systems  $u_{t_i} = K_i\cdot u_{x}$ and $u_{t_j} = K_j\cdot u_{x}$  on $u(x, t_i, t_j)$ are compatible). In particular, under the assumption that the operators are semi-simple,   recent results  from \cite{BKM2026} combined with those of   \cite{tsarev} imply that 
the   system  \eqref{eq:K0} can be solved  in quadratures. See also  \cite{LPG2024} for a   discussion of the non-semisimple case. 

The  application of operator  Frobenius algebras in the theory of integrable systems of hydrodynamic type is based on the following construction. For an $n$-dimensional operator Frobenius algebra $\mathcal K=\textrm{Span}(K_1,\dots, K_n)$ we construct,  by a purely algebraic procedure,  an operator Frobenius algebra $\mathcal M$ which we call  \emph{dual}  to $\mathcal K$. We show in Theorem \ref{t2} that if the operators $K_i$ and $K_j$ are mutual symmetries for all $i,j=1,\dots, n$, then any two elements of $\mathcal M$ are mutual symmetries also.  

Under the assumption that one of the operators $K_1,\dots, K_n$ is the identity operator, the construction enjoys the standard property of  {\it duality} in the sense that the dual algebra to $\mathcal M$ is $\mathcal K$, that is, $(\mathcal K^*)^*=\mathcal K$. Moreover, in this case we express common symmetries of $\mathcal K$ in terms of common conservation laws of $\mathcal M$, and vice versa. As a byproduct we show that   in   the real-analytic case, any operator Frobenius algebra   containing the identity and
generated by mutual symmetries admits $n$ independent conservation laws.

Our construction allows one to provide many new examples of integrable systems of hydrodynamic type, in all dimensions, such that the generator $K$ of the corresponding system \eqref{eq:K0} has Jordan blocks of arbitrary dimension. Only a few such examples are known in literature. See the  introduction to  \cite{LPG2024} for a list of  examples, some of which come from mathematical physics.  In fact, even if we start with something fairly simple, say commuting Nijenhuis operators,  the generators of the dual integrable systems will be quite nontrivial and not Nijenhuis operators anymore. 
In the diagonal case, the obtained systems are not new. They have been studied, e.g., in \cite{FeFo,BlMa}; in the presence of Jordan block(s) most of them are new.

The second  application, which we view as the main result of the paper, relates to the theory of finite-dimensional integrable systems. We study Hamiltonian systems on the cotangent bundle $T^*\mathsf M^n$ defined by $n$ Poisson commuting functions $F_1,\dots,F_n$ that are quadratic in momenta $p$ with coefficients depending on the position. We assume that one of them is nondegenerate, so that one can view it as a pseudo-Riemannian metric  $g=(g^{ij})$. Then, the Poisson commuting functions $F_\alpha$ have the form
\begin{equation}\label{eq:formmetric}
F_\alpha:T^*\mathsf M^n\to \mathbb{R}, \ \  F_\alpha(p,u) = \tfrac{1}{2} g(K_\alpha^* p, p)   \ \  \alpha=1,\dots, n.
\end{equation}
Here $K_\alpha=K_\alpha(u)$ are operator fields self-adjoint with respect to the metric $g=g(u)$, with $K_1=\Id$. We additionally assume that the operator fields commute in the algebraic sense, i.e., $K_i\cdot K_j=K_j\cdot K_i$,    and $\det \left(  \frac{\partial F_\alpha}{\partial p_i} (u,p) \right)\ne 0$ for generic $p\in T_u \mathsf M$.

Such a situation is well understood in the case when the operators $K_\alpha$ are diagonalisable. It is closely related to orthogonal separation of variables which is an effective and widely used tool in mathematical physics, with first results obtained already by C. Jacobi \cite{Jacobi}.  A milestone result in this direction was made by P. St\"ackel in 1891 \cite{staeckel}, who constructed an example of such integrable systems in all dimensions. The St\"ackel  construction is explicit, and the freedom is the choice of $n(n+1)$ functions of one variable. The operators $K_\alpha$ coming from it are simultaneously diagonalisable and therefore commute.

L. Eisenhart in 1934 \cite{Eisenhart34} asked the question under which conditions the existence of Poisson commuting integrals of the form \eqref{eq:formmetric} would imply that the integrals and the metric come from the St\"ackel construction, and gave first results in this direction.   The question was answered in full generality by E. Kalnins and W. Miller  in \cite{KM1980}   (see also \cite{agafonov}), where it was  proved that functionally independent Poisson commuting integrals  \eqref{eq:formmetric} necessarily come from the St\"ackel construction if one assumes that at every point $\mathsf p\in\mathsf M$ there exists a basis in $T_{\mathsf p}\mathsf M^n$ in which the operators $K_\alpha$ are all diagonal. 

It is an interesting, natural and actively studied (in particular in the framework of general relativity, see e.g.  \cite{KoPa}) problem to generalize the results of 
St\"ackel,  Kalnins
and Miller by merely requiring that  the operators $K_\alpha$ commute in the algebraic sense.

In our paper we solve this problem in full generality  in the following sense.  In Theorem \ref{t5} we describe a construction of such quadratic-in-momenta functions $F_\alpha$.  The  first step  of the construction is a choice 
of an   $n$-dimensional operator  Frobenius  algebra $\mathcal M=\textrm{Span}(M^1,\dots, M^n)$ of mutual symmetries consisting of Nijenhuis operators.  The second ingredient is a common conservation law $\alpha$  of $M^i$'s such that the differential forms $M^{i*}\alpha$ are linearly independent at every point.  Theorems \ref{t3} and  \ref{t4}   describe properties of such operator algebras, explain how to construct examples of $\mathcal M$ and why $\alpha$ exists.     
Then, Theorem  \ref{t5} constructs, by an algebraic formula involving $M^i$ and $\alpha$, 
Poisson commuting independent functions $F_1,\dots, F_n$ of the  form \eqref{eq:formmetric}.

Next, in Theorem \ref{t6}, we prove that any nondegenerate  collection of Poisson commuting independent functions of the form \eqref{eq:formmetric} with algebraically commuting operators $K_\alpha$ comes from the construction in Theorem \ref{t5}. 

Let us now explain how the solution of this generalised St\"ackel--Eisenhart--Kalnins--Miller problem is related to operator Frobenius algebras and their duality. Our proof consists of two steps. First, we show that the operators $K_\alpha$ are mutual symmetries and generate an operator Frobenius algebra, which naturally contains the identity operator. Next we show that the dual of this operator algebra consists of Nijenhuis operators. This correspondence also works in the other direction: for an $n$-dimensional symmetry subalgebra and a choice of a regular common conservation law, which fixes a bilinear form of the Frobenius metric and therefore a metric on the manifold, we obtain Poisson commuting integrals. Note that although the scheme of the proof is quite simple, the proof is nontrivial in both directions, with the duality of operator Frobenius algebras playing the major role.

The paper is organized as follows:
\begin{itemize}
    
    \item Section \ref{s1} is purely algebraic: it deals with the algebraic aspects of our construction. We introduce the notion of regular representation for a commutative Frobenius algebra and provide the criterion for a given subspace $\mathfrak A \subset \mathfrak{gl}(\mathbb K^n)$ to be the image of such a representation. Next, we introduce the notion of dual bases in a  Frobenius algebra. 
    \item Section \ref{s2} is devoted to the geometric properties of dual Frobenius operator algebras $\mathcal K=\Span_{\R}(K_1,\dots,K_n)$ and $\mathcal M = \Span_{\R}(M^1,\dots,M^n)$. We show that, under the condition that $K_i$'s are mutual symmetries, the same property holds for $M^j$'s and, moreover, there is a natural correspondence between common symmetries of $K_i$'s and common conservation laws of their dual $M^i$'s and vice versa.
    
    \item Section \ref{s3} is dedicated to the  theory of symmetry algebras $\Sym=\Sym(\mathcal M)$  for Frobenius operator algebras $\mathcal M =\Span_{\R}(M_1,\dots, M^n)$ whose elements are mutual symmetries and, in addition, are all Nijenhuis operators. We show that $\Sym$ is an 
    infinite-dimensional commutative subalgebra of the algebra of operator fields on $\mathsf M^n$. We describe it explicitly, study common conservation laws for all elements of $\Sym$, and introduce the important notion of flat symmetry algebras.
    
    \item In Section \ref{s4}, we first  present a construction of finite-dimensional integrable systems such that the corresponding integrals are quadratic in momenta. The Killing $(1,1)$-tensors $K_1,\dots, K_n$ corresponding to these integrals algebraically commute and are generic in the sense that the vectors $K_1\xi, \dots, K_n\xi$ are linearly independent for a generic tangent vector $\xi$.  Our method uses the symmetry algebras $\Sym=\Sym(\mathcal M)$ from Section \ref{s3} and is a generalization of the construction  given in \cite{st}. Then we show that every integrable system of this type with algebraically commuting Killing $(1,1)$-tensors is obtained via our  construction. This yields the solution to the St\"ackel-Eisenhart-Kalnins-Miller problem in the general (not necessarily diagonal) case.
    
\end{itemize}

\noindent{\bf Acknowledgments.}  A.\,B. was       supported by the Ministry of Science and Higher Education of the Republic of Kazakhstan (grant No. AP23483476), and   V.\,M. by  the DFG (project  529233771) and the ARC Discovery Program DP210100951.


\section{Frobenius algebras and their regular representations} \label{s1}

In this paper,  a Frobenius algebra $(\mathfrak a,\star)$ is a finite-dimensional unital commutative associative algebra $\mathfrak a$ over the field $\mathbb K = \mathbb R$ or $\mathbb C$ possessing a   nondegenerate symmetric bilinear form $b(\cdot, \cdot)$ such that
$$
b(\xi \star \eta, \zeta)= b(\eta, \xi\star\zeta).
$$
The form $b$ is not unique.  The only important property is that such a form exists, and the choice of $b$, in our paper, will depend on the situation.
 
For an arbitrary $\xi \in \mathfrak a$, define the operator of (right) multiplication $\mathsf R_\xi: \mathfrak a \to \mathfrak a$ as 
$$
\mathsf R_\xi \eta = \eta \star \xi. 
$$
The correspondence  $\xi \to \mathsf R_\xi$ can be understood as a representation of $\mathfrak a$:
$$
\mathsf R : \mathfrak a \to \gl(\mathfrak a),
$$
 known as the {\it regular} representation. Notice that the regular representation is faithful since for the unity $e\in \mathfrak a$ we have $\mathsf R_\xi e = \xi$ so that $\mathsf R_\xi=0$ if and only if $\xi=0$. Hence the image of $\mathsf R$ is a commutative subalgebra in  $\gl(\mathfrak a)$ isomorphic to $\mathfrak a$. 

Now let $V$ be an $n$-dimensional space over $\mathbb R$ or $\mathbb C$ and $K_1, \dots, K_n\in \gl(V)$ be $n$ linearly independent  commuting operators. 
The following theorem provides necessary and sufficient conditions for  $\mathfrak A=\Span(K_1, \dots, K_n)$ to be the image of the regular representation of a certain $n$-dimensional Frobenius algebra $\mathfrak a$ (under a suitable identification  $\mathfrak a \simeq V$).  

\begin{Theorem}\label{t1} Let $K_1,\dots, K_n: V \to V$, $\dim V=n$, be commuting operators. 
 The subspace $\mathfrak A=\Span(K_1, \dots, K_n) \subset \gl (V)$  is the image of the regular representation of an $n$-dimensional Frobenius algebra $\mathfrak a$  (under a certain identification $\mathfrak a \simeq V$) if and only if the following two conditions hold:
\begin{itemize}
\item[\rm(A1)] There exists $\xi\in V$ such that  
$$
\Phi_\xi :  \mathfrak A \to V,  \quad \Phi_\xi (M) = M\xi,
$$ 
is an isomorphism of vector spaces or, equivalently, $K_1\xi, \dots, K_n\xi$ are linearly independent. 

\item[\rm(A2)]  There exists $a\in V^*$ such that 
$$
\Phi^*_a :  \mathfrak A \to V^*,  \quad \Phi^*_a (M) = M^*a,
$$
 is an isomorphism of vector spaces or, equivalently, $K_1^*a, \dots, K_n^*a$ are linearly independent.
 \end{itemize}
\end{Theorem}
 
 \begin{proof}
 Let us show that the first condition (A1) is equivalent to the fact that $\mathfrak A=\Span(K_1,\dots,K_n)$ is the image of some unital commutative associative algebra $\mathfrak a$ under the regular representation.
 
 First assume that the latter property holds. Then the unity $e\in \mathfrak a \simeq V$ satisfies condition (A1).   Indeed,  if $M = \mathsf R_\xi \in \operatorname{Ker} \Phi_e$,  $\xi\in \mathfrak a$, then $0 = \Phi_e (M)  = M e = \xi \star e = \xi$, that is, $M=\mathsf R_\xi=0$.   Hence the kernel of $\Phi_e$ is trivial, and $\Phi_e$ is an isomorphism as stated. 
 
 Conversely, assume that $\Phi_\xi : \mathfrak A \to V$ is an isomorphism.  
 
 \begin{Lemma}
 \label{lem:2.1}
 Let $K_1, \dots, K_n$ commute and satisfy {\rm (A1)}. Then $\mathfrak A=\Span(K_1, \dots, K_n)$  coincides with its own centraliser   $\mathcal C(\mathfrak A) = \{ B \in \gl(V)~|~ B{\cdot}A = A{\cdot}B \ \mbox{for all $A\in \mathfrak A$}\}$\footnote{Throughout the paper we use $A\cdot B$ for the usual product of matrices and linear operators.} and, therefore, is automatically a unital commutative associative algebra.  
 \end{Lemma}
 \begin{proof}
Take $B\in \mathcal C( \mathfrak A)$.  Since $\Phi_\xi$ is an isomorphism, there exists $A\in \mathfrak A$ such that  $\Phi_\xi A = B\xi$, i.e., $A\xi = B\xi$.  Also, for any other vector $\eta \in V$ there exists $\widetilde   A \in  \mathfrak A$ such that $\eta = \Phi_\xi \widetilde   A = \widetilde   A \xi$. Then
$$
B \eta = B (\widetilde   A \xi) = \widetilde   A (B \xi) = \widetilde   A ( A\xi) = A(\widetilde   A \xi) = A \eta.
$$
In other words,  $B \eta  =  A \eta$ for any $\eta \in V$, that is, $B = A \in \mathfrak A$.  
Thus,  $\mathfrak A = \mathcal C( \mathfrak A)$ and $\mathfrak A \subset \gl(V)$ is a unital commutative associative (sub)algebra.
\end{proof}

The isomorphism $\Phi_\xi: \mathfrak A \to V$ allows us to identify the algebra $\mathfrak A$ with $V$.   This identification allows us to endow  $V$ with the structure of a unital associative commutative algebra with the naturally induced operation $\star$ on $V$ 
$$
\eta \star \zeta = \Phi_\xi (A{\cdot}B), \quad \mbox{where $A=\Phi_\xi^{-1} \eta$, $B=\Phi_\xi^{-1} \zeta$}.
$$
Notice that $\Phi_\xi (\Id) = \Id \, \xi = \xi$, 
i.e., under this identification,  $\xi \in V$ becomes the unity element. 
 
It remains to notice that the inverse map  $\Phi_\xi^{-1}$
$$
\eta \in V  \quad  \mapsto \quad  \Phi_\xi^{-1} \eta \in \gl (V)
$$      
coincides with the regular representation of $(V, \star)$, that is,   $\Phi_\xi^{-1} (\eta) = \mathsf R_\eta$. In other words, $\mathfrak A\subset \gl(V)$ coincides with the image of $(V, \star)$ under the regular representation.  This statement is almost tautological. Indeed,   let  $\Phi_\xi^{-1} (\eta) = A$ and $\Phi_\xi^{-1} (\zeta) = B$, or equivalently, $\eta = A\xi$, $\zeta = B\xi$.  Then
$$
\mathsf R_\eta \zeta = \eta \star \zeta = \Phi_\xi (A{\cdot}B) = A(B\xi) = A \zeta = \Phi_\xi^{-1} (\eta) \zeta , \quad \mbox{that is $\Phi_\xi^{-1} (\eta)=\mathsf R_\eta$},
$$ 
as stated.   

Let us now show that the second property (A2) is equivalent to the existence of a Frobenius form, assuming that (A1) holds so that we may use the identification $\mathfrak A\simeq V$.

If $\mathfrak A$ is Frobenius,  then we can use the Frobenius form $b(\cdot,\cdot)$ to identify $\mathfrak A\simeq V$ with its dual  $\mathfrak A^*\simeq V^*$ and the existence of $a \in V^*$ satisfying (A2) immediately follows since conditions (A1) and (A2) become  identical.

Conversely, assume that there exists $a \in V^*$ such that the map $M \in \mathfrak A  \mapsto M^*a \in  {V}^*$ is an isomorphism of vector spaces.  Consider the following bilinear form on $(V, \star)$:
$$
b( \eta, \zeta) = \langle a \, ;  \eta\star\zeta  \rangle,  
$$
where $\langle~ \, ;~\rangle$ denotes pairing between vectors and covectors.
This form is obviously invariant, and we only need to check that $b(\cdot,\cdot)$ is nondegenerate.  Take an arbitrary $0\ne\eta\in 
V$ and let $M\in \mathfrak A$ be such that $M\xi = \eta$, or equivalently, $M = \Phi^{-1}_\xi (\eta)$.  Then
$$
b(\eta, \zeta) = \langle a \, ;  \eta\star\zeta  \rangle = \langle a\, ; M\zeta \rangle = \langle M^*a\, ; \zeta\rangle.
$$
Since $M \mapsto M^*a$ is an isomorphism, then $M^*a \ne 0$ and therefore there exists $\zeta$ such that $b(\eta, \zeta)\ne 0$.  This shows that $b(\cdot, \cdot)$ is a nondegenerate invariant form, i.e., $(V, \star)$ is Frobenius.  In other words, $\mathfrak A$ is the image of the regular representation of the Frobenius algebra $(V,\star)$, as required.  This proves  Theorem \ref{t1}.
\end{proof}

\begin{Remark}\label{p1}
{\rm
Let $\mathfrak A \subset \gl(V)$, $\dim V=n$,  be an $n$-dimensional Frobenius algebra satisfying conditions (A1) and (A2) and $\xi_0\in V$ be a certain vector from condition (A1).  As explained, we can use 
the map $\Phi_{\xi_0} : \mathfrak A \to V$ to identify $V$ with $\mathfrak A$.   Under this identification, $\xi_0\in V\simeq \mathfrak A$ becomes the unity element.  However, this vector can be replaced by any other vector $\xi$ such that $K_1\xi, \dots , K_n\xi$ are linearly independent, where $K_1,\dots,K_n$ is a basis of $\mathfrak A$.  The set $V_\circ \subset V$ of such generic vectors $\xi\in V$ coincides with the set of invertible elements of $\mathfrak A \simeq V$. Notice also that the group of invertible operators $\mathfrak A_\circ \subset \mathfrak A$ acts on $V_\circ$ freely and transitively.  Similarly,  if we denote by $(V^*)_\circ\subset V^*$ the set of generic covectors  $a\in V^*$ such that $K^*_1 a,\dots, K^*_n a$ are linearly independent, then the group $\mathfrak A_\circ$ acts on $(V^*)_\circ$ freely and transitively.  Of course,  $V_\circ$ and $(V^*)_\circ$ are isomorphic and are mapped to each other under identification of $V\simeq \mathfrak A$ and $V^*\simeq\mathfrak A^*$ by means of any Frobenius form $b$.
}\end{Remark}
 
In what follows, we work with Frobenius subalgebras $\mathfrak A \subset \gl(V)$ satisfying (A1) and (A2) with $V=T_x\mathsf M$ being the tangent space of a smooth manifold. One of the main concepts used below is Frobenius duality/dual bases.

Recall that every Frobenius form $b : \mathfrak A\times \mathfrak A \to \mathbb K$ is defined by means of a generic covector $a\in \mathfrak (\mathfrak A^*)_\circ \subset \mathfrak A^*$:
$$
b(K, M) = b_a(K, M) = \langle a \, ; K{\cdot}M \rangle, \quad K,M\in \mathfrak A \subset \gl (V).
$$
We will say that two bases $K_1,\dots, K_n$ and $M^1,\dots, M^n$ of $\mathfrak A$ are  {\it dual}  with respect to  
$b=b_a$, if 
$$
b_a( K_i, M^j ) =  \delta^j_i=\begin{cases} 1,  &\mbox{if} \ i=j \\ 0,  &\mbox{if}\  i\ne j  \end{cases} \quad \mbox{or, equivalently,} \quad \langle a\, ; K_i{\cdot}M^j\rangle = \delta^i_j.
$$

Finally, let us discuss  some elementary tensor-type formulas related to Frobenius duality.  Let $a_{ij}^k$ be the structure constants of 
$\mathfrak A$ in the basis $K_1,\dots, K_n$, that is,
$$
K_i{\cdot} K_j = a_{ij}^m K_m. 
$$ 
Then the components of the Frobenius form $b=b_a$ are given by
$$
b_{ij} = a_{ij}^m a_m,
$$ 
where $a_m = a(K_m)$ are the components of the covector $a$ w.r.t. the basis $K_1,\dots, K_n$.
The vectors of the dual bases are related as follows
\begin{equation}
\label{eq:n18}
K_i = b_{ij} M^j \quad\mbox{or, equivalently,}\quad M^j= b^{ij}K_i,
\end{equation}
where $b^{ij}$ are the components of $b$ in the basis $M^1,\dots, M^n$ or, which is the same, the elements of the inverse matrix to $b_{ki}$ so that $b_{ki}b^{ij}= \delta^j_k$.

The structure constants $a^{ij}_k$ of $\mathfrak A$ in the basis $M^1,\dots, M^n$ (defined from the standard relations
$M^i{\cdot} M^j = a^{ij}_k M^k$)  can also be obtained from $a_{ij}^k$ by raising-lowering indices: 
$$
a^{ij}_k = b^{i\alpha}b^{j\beta} b_{k\gamma} a_{\alpha\beta}^\gamma.
$$
This formula can be simplified, if we use   the invariance of $b$, which can be written as  $b_{k\gamma} a_{\alpha\beta}^\gamma = b_{\alpha\gamma} a_{k \beta}^\gamma$. Namely,
\begin{equation}
\label{eq:n17}
a^{ij}_k = b^{i\alpha}b^{j\beta} b_{k\gamma} a_{\alpha\beta}^\gamma =
b^{i\alpha}b^{j\beta} b_{\alpha\gamma} a_{k \beta}^\gamma = \delta^i_\gamma b^{j\beta} a_{k\beta}^\gamma =
b^{j\beta} a_{k\beta}^i
\end{equation}

Finally we notice that the components $a^k = \langle a\, ; M^k\rangle=b^{kj} a_j$ of the covector $a$ in the dual basis coincide with the coordinates of the identity operator $\Id$ in the basis $K_1, \dots, K_n$, that is,
$$
\Id = a^1 K_1 + \dots + a^n K_n \quad\mbox{and, by duality,}\quad\Id = a_1M^1 + \dots + a_n M^n.
$$
Indeed, if $\Id = \sum \beta^i K_i$, then
\begin{equation}
\label{eq:n1}
a^j =\langle a\, ; M^j\rangle = \langle a\, ; M^j{\cdot}\Id\rangle = b (M^j , \mbox{$\sum \beta^i K_i$}) = \beta^i b(M^j, K_i) = \delta^j_i \beta^i = \beta^j.
\end{equation}

We finish this section with two examples of Frobenius subalgebras $\mathfrak A\subset \gl(V)$ satisfying properties (A1) and (A2). 

\begin{Ex}\label{3:ex1}{\rm Consider a $\gl$-regular linear operator $L: V \to V$, $\dim V = n$.  The $\gl$-regularity condition can be introduced in several equivalent ways (see e.g. \cite{nij3}).  For instance, we may require that $n$ operators $\Id, L, \dots, L^{n-1}$ satisfy condition (A1), that is, there exists a (cyclic) vector $\xi\in V$ such that $\xi, L\xi, \dots, L^{n-1}\xi$ form a basis of $V$.   Then the centraliser of $L$
$$
\mathcal C(L) =\{ M\in \gl (V)~|~ M{\cdot}L=L{\cdot}M\,\} = \Span(\Id, L, L^2,\dots, L^{n-1}) 
$$ 
is a Frobenius subalgebra in $\gl(V)$ satisfying  (A1) and (A2). Typical examples of such subalgebras are the centralisers of 
a diagonal matrix with simple spectrum and of a Jordan block, which have respectively the following form (in dimension 4):
$$
\mathfrak A_{\mathsf{diag}} = 
\left\{\begin{pmatrix} u_1 & 0 &0 &0 \\ 0 & u_2 & 0 & 0 \\
0 & 0 & u_3 & 0 \\
0 & 0 & 0 & u_4
\end{pmatrix}, \ u_i\in\mathbb K\right\}
,\qquad
\mathfrak A_{\mathsf{Jordan}}=
\left\{\begin{pmatrix} u_1 & 0 & 0 & 0 \\ u_2& u_1 & 0 & 0 \\
u_3& u_2 & u_1 & 0 \\
u_4 & u_3 & u_2 & u_1
\end{pmatrix}, \ u_i\in\mathbb K\right\}.
$$
}\end{Ex}

In dimension $n\le 3$, every Frobenius algebra is isomorphic to the centraliser of a certain $\gl$-regular operator $L\in\gl(n,\mathbb K)$.  Starting from dimension 4, there exist essentially different examples.

\begin{Ex} \label{ex:3.2}
\rm{
 The 4-dimensional subalgebra of $\gl(4,\mathbb K)$
 \begin{equation}
 \label{eq:n16}
\mathfrak A = \left\{
\begin{pmatrix} u_1 & 0 & 0 & 0 \\ u_2& u_1 & 0 & 0 \\
u_3& 0 & u_1 & 0 \\
u_4 & u_2 & u_3 & u_1
\end{pmatrix}, \ u_i\in\mathbb K\right\}
\end{equation}
is Frobenius and satisfies conditions (A1) and (A2).  This algebra, however, contains no $\gl$-regular operators. 
}
\end{Ex}


\section{ Operator Frobenius algebras and their duals}\label{s2}

In this section, we consider $n$ commuting operator fields $K_1,\dots, K_n$ on a smooth manifold $\mathsf M$, $\dim \mathsf M=n$, such that at each point $x\in \mathsf M$, they span an $n$-dimensional Frobenius subalgebra $\mathfrak A(x)\subset \gl(T_x\mathsf M)$ satisfying properties (A1) and (A2) discussed in the previous section.  In other words, $\mathfrak A(x)$ is the image of the regular representation of a certain Frobenius algebra.  We will refer to $\mathcal K = \operatorname{Span}_{\R}(K_1,\dots, K_n)$ as {\it an operator Frobenius algebra}.

As before, we introduce the structure tensor $a_{ij}^s$ in the standard way
$$
K_i{\cdot}K_j = \sum_s a_{ij}^s K_s.
$$ 
Notice that $a_{ij}^s=a_{ij}^s(x)$ are smooth functions on $\mathsf M$. Thus, if we do not fix a point $x\in\mathsf M$, then $\operatorname{Span}_{\R}(K_1,\dots, K_n)$ is not a Frobenius algebra anymore unless $a_{ij}^s$ are constant.

 Keeping the notation from the previous section, we now literally repeat the construction of the dual basis for operator fields.  Let us choose constants $a_1,\dots, a_n\in \R$ such that the symmetric bilinear form $b_{ij}(x) = a_{ij}^k(x) a_k$ is non-degenerate.  These constants define an element $a$ from the dual space   $\mathcal K^*=\Span^*_{\R}(K_1,\dots, K_n)$ such that $\langle a\, ; K_i\rangle = a_i$.  This form $b=b_a$ serves as a Frobenius form at each tangent space $T_x\mathsf M$ so that all the algebraic relations discussed in Section \ref{s1} hold.  Let us emphasise that the objects $a_{ij}^k$, $b_{ij}$ and $a_k$ should not be treated as tensor fields on $\mathsf M$  (the underlying vector space for them is  $\mathfrak A(x)$, but not the tangent space $T_x\mathsf M$).  On the other hand,  the bilinear form $b$ and linear function $a$ can be applied to arbitrary linear combinations of operators $K_i$ with functional coefficients.  
 
We now use $b=b_a$ to construct dual operator fields $M^1,\dots, M^n$, $M^i(x)\in \mathfrak A(x)$:
$$
b_a ( M^i, K_j) \overset{\mathrm{def}}{=}\langle a \, ; M^i\cdot K_j\rangle=\delta^i_j, \quad\mbox{or, explicitly,} \quad M^i = b^{ij} K_j. 
$$   
Notice that $\mathcal M = \operatorname{Span}_{\R}(M^1,\dots, M^n)$ can be naturally identified with $\mathcal K^*$ since for any $M\in\mathcal M$ and $K\in\mathcal K$ we have $b_a(M,K)\in \R$.  This identification depends, however, on the choice of the  covector $a\in \mathcal K^*$.  To emphasise this point, we will write $\mathcal M = \mathcal K^*_a$.

In the third statement of Theorem \ref{t2} below, we will additionally assume that 
\begin{equation}
\label{eq:Id_assumption}
\Id \in \mathcal K = \operatorname{Span}_{\R}(K_1,\dots, K_n). 
\end{equation}
In other words, there exist constants $\beta^1,\dots,\beta^n\in\R$ such that $\Id = \sum \beta^i K_i$.  This assumption ensures that the relation between  the operator Frobenius algebras $\mathcal K$ and $\mathcal M$  is symmetric. 
Indeed, we observe from \eqref{eq:n1} that $a^i  \overset{\mathrm{def}}{=} \langle a \, ; M^i\rangle=\beta^i=\mathrm{const}\in\R$.  Hence, the covector $a$ used for constructing the Frobenius form $b_a$ can be treated as an element from $\mathcal M^*=\Span^*_{\R} (M^1,\dots, M^n)$\footnote{Otherwise,  $a^i = \langle a; M^i\rangle$ is not necessarily constant, and the `inclusion' $a\in \Span^*_{\R} (M^1,\dots, M^n)$ makes no sense.}. 
This implies that repeating the above construction for $M^1,\dots, M^n$ we return to the original basis $K_1,\dots, K_n$, that is,  between these collections of operators there is a natural {\it geometric duality} which takes into account the fact that all of our geometric objects are not constants, but essentially depend on a point $x\in\mathsf M$. In particular, $\mathcal M^*_a =\mathcal K$. In this context, it is important to remember that $\mathcal K=\Span_{\R} (K_1,\dots, K_n) \ne \mathcal M=\Span_{\R} (M^1,\dots,M^n)$ despite the fact that $\Span_{\R} \bigl(K_1(x),\dots, K_n(x)\bigr) = \Span_{\R} \bigl(M^1(x),\dots,M^n(x)\bigr)=\mathfrak A(x)$ for each $x\in\mathsf M$. 

So far, the construction was purely algebraic in the sense that we did not impose any geometric conditions on the operator fields involved. Let us now introduce a differential-geometric condition on the operator fields.

Let $L, M$ be a pair of operator fields. For a pair of vector fields $\xi, \eta$, define
\begin{equation}\label{ii:eq2}
\langle L, M \rangle (\xi, \eta) = LM[\xi, \eta] + [L\xi, M\eta] - L[\xi, M\eta] - M[L\xi, \eta].
\end{equation}
Here, square brackets stand for the standard commutator of vector fields.

The r.h.s. of  \eqref{ii:eq2}  defines a tensor field of type $(1, 2)$ if and only if the operators commute, that is, if  $LM - ML = 0$  by 
\cite[formula 3.9]{nijenhuis}. We say that  commuting $L$ and $ M$ are \emph{symmetries} of each other if the symmetric (in lower indices) part of $\langle L, M \rangle$ vanishes. If the entire tensor vanishes, then $L, M$ are called \emph{strong symmetries}. By  \cite{st},   if  $R, M$ are strong symmetries  of $L$ then  $RM$ is a strong symmetry of $L$. Note that $\Id$ is a strong symmetry for any operator field $L$.

A closed 1-form $\alpha$ is said to be a \emph{conservation law} of $L$ if $L^* \alpha$ is also closed. Our statements are local, so we assume that $\alpha$ is exact; that is, $\alpha = \ddd f$. The function $f$ is referred to as the density of the conservation law. Above $L^*$ denotes the operator dual to $L$, that is,  acting on $T^*\mathsf M$   by $L^* \alpha(\xi)= \alpha(L\xi)$.

The next theorem describes the relationship between dual operator Frobenius algebras $\mathcal K$ and $\mathcal M$ under the assumption that $K_1,\dots, K_n\in\mathcal K$ are symmetries of each other.

\begin{Theorem}\label{t2} Let $\mathcal K$ and $\mathcal M = \mathcal K_a^*$ be dual operator Frobenius algebras in the above sense w.r.t. the form $b=b_a$ and $K_1,\dots, K_n\in \mathcal K$  and $M^1,\dots, M^n\in\mathcal M$ be the corresponding dual bases.   Assume that $K_1,\dots,K_n$ are symmetries of each other, i.e., $\langle K_i, K_j\rangle (\xi,\xi)=0$, $i,j=1,\dots,n$.  Then the following is true: 
\begin{itemize}

\item[$(1)$]  $M^1,\dots, M^n$ are symmetries of each other.

\item[$(2)$]   If $K=h^1(x)K_1+\dots + h^n(x)K_n$ is a common symmetry of $K_1,\dots,K_n$ and $h=a_s h^s$, then $\ddd h$ is a common conservation law of $M^1,\dots, M^n$ and ${M^i}^*\ddd h  = \ddd h^i$.  Conversely, if $\ddd h$ is a common conservation law of $M^1,\dots, M^n$ and ${M^i}^*\ddd h  = \ddd h^i$, then  $K=h^1(x)K_1+\dots + h^n(x)K_n$ is a common symmetry of $K_1,\dots,K_n$.

\item[$(3)$]  Assume in addition that \eqref{eq:Id_assumption} holds, that is $\Id = \sum \beta^i K_i$ for some $\beta^i\in\R$.  If an exact 1-form $\ddd f$ is a common conservation law for $K_1,\dots, K_n$, then the operator $M=f_1(x) M^1+\dots+f_n(x)M^n$ defined from the relation $K_i^*\ddd f = \ddd f_i$
is a common symmetry of $M^1,\dots, M^n$.  Conversely, if $M=f_1(x) M^1+\dots+f_n(x)M^n$  is a common symmetry of $M^1,\dots, M^n$ and $f=\sum \beta^i f_i$, then $\ddd f$ is  a common conservation law for $K_1,\dots, K_n$ and $K_i^*\ddd f = \ddd f_i$.

\end{itemize}

\end{Theorem}

  Note that locally the functions $h^i$  are determined up to
additive constants. Consequently $\ddd h$ determines the symmetry modulo the constant span
of $K_1,\dots,K_n$.  

\begin{proof}
We start with the second statement. 

\begin{Lemma}\label{dfg1}
An operator field $K = h^1 K_1 + \dots + h^n K_n$ is a common symmetry of $K_i$'s if and only if the differentials $\ddd h^i$ satisfy the following linear system
\begin{equation}\label{epic1}
   K_i^* \ddd h^j = a^j_{is} \ddd h^s.  
\end{equation}
\end{Lemma}
\begin{proof} 

By direct computation, we have 
\begin{equation}\label{top4}
\begin{aligned}
\langle K_i, K \rangle (\xi, \eta) & = [K_i \xi, K_s h^s \eta] + K_i K_s h^s [\xi, \eta] - h^s K_s [K_i\xi, \eta] - K_i [\xi, h^s K_s \eta] = \\
& = h^s \langle K_i, K_s \rangle (\xi, \eta) + \ddd h^s (K_i \xi) \cdot K_s \eta - \ddd h^s (\xi) \cdot K_s K_i \eta = \\
& = h^s \langle K_i, K_s \rangle (\xi, \eta) + K_i^*\ddd h^j(\xi) \cdot K_j \eta - a_{is}^j  \ddd h^s(\xi)  \cdot  K_j  \eta =\\
& = h^s \langle K_i, K_s \rangle (\xi, \eta) + \left(K_i^* \ddd h^j(\xi) -  a_{is}^j\ddd h^s (\xi)\right) \cdot K_j  \eta.
\end{aligned}    
\end{equation}

Taking $\xi = \eta$ to be a generic vector (i.e., such that $K_1\xi, \dots, K_n\xi$ are linearly independent), we observe that $\langle K_i, K_s \rangle (\xi, \xi) = 0$ and, due to linear independence of $K_j\xi$, conclude that $\langle K_i,  K \rangle (\xi,\xi)=0$ if and only if
$$
K_i^* \ddd h^j (\xi) - a^j_{is} \ddd h^s(\xi) = 0.
$$
Since $\xi$ is generic,  $\langle K_i, K_s \rangle (\xi, \xi)$ vanishes identically for all $\xi$ if and only if $K_i^* \ddd h^j  - a^j_{is} \ddd h^s =0$, as required.
\end{proof}

It remains to show that \eqref{epic1} is equivalent to ${M^j}^* \ddd h = \ddd h^j$. First, assume that \eqref{epic1} holds.  Then taking into account that $h =\sum_j a_j h^j$, we have 
$$
K_i^* \ddd h = K_i^* \Big( \sum_{j = 1}^n a_j \ddd h^j \Big) = \sum_{j = 1}^n a_j K_i^* \ddd h^j = a_j a^j_{is} \ddd h^s = b_{is} \ddd h^s.
$$
Hence
$$
{M^j}^* \ddd h = b^{ij} K_i^* \ddd h = b^{ij} b_{is}\ddd h^s = \delta^j_s \ddd h^s = \ddd h^j,
$$
as required.

Conversely, assume that ${M^j}^* \ddd h = \ddd h^j$.  Then
$$
\ddd h^j = {M^j}^* \ddd h = b^{ij} K_i^* \ddd h \quad\mbox{or, equivalently,}\quad K_i^* \ddd h =  b_{im} \ddd h^m = a_k a^k_{im} \ddd h^m.
$$
Next, we have
$$
\begin{aligned}
K_i^* \ddd h^j &= K_i^* M^{j*} \ddd h = M^{j*} K_i^*\ddd h =  
 {M^j}^* b_{im}  \ddd h^m = \\  &=
b_{im} {M^j}^* {M^m}^* \ddd h =   b_{im} a^{jm}_s {M^s}^* \ddd h = b_{im} a^{jm}_s \ddd h^s = a_{is}^j \ddd h^s,
\end{aligned}
$$
as required.   

The third statement follows from the fact that $\mathcal K = \mathcal M^*_a$. It is sufficient to interchange $\mathcal M$ and $\mathcal K$ and apply the second statement. To prove the first statement, we need two lemmas.

\begin{Lemma}\label{dfg3}
Let two operator fields $L$ and $M$ admit $n$ linearly independent common conservation laws $\ddd f_1, \dots, \ddd f_n$ and $LM=ML$. Then
\begin{enumerate}
    \item $L$ and $M$ are symmetries of each other.
    \item If, in addition, $\ddd f_1, \dots, \ddd f_n$ are  conservation laws of $LM$, then $L$ and $M$ are strong symmetries of each other.
\end{enumerate}
\end{Lemma}
\begin{proof} Take $u^i = f_i$ as local coordinates. Since the 1-forms $L^*\ddd u^i$ and $M^* \ddd u^i$ are closed, in these coordinates  we obtain:
$$
\pd{L^i_j}{u^k} - \pd{L^i_k}{u^j} = 0, \quad \pd{M^i_j}{u^k} - \pd{M^i_k}{u^j} = 0.
$$
Applying these identities to the explicit formula of the tensor $\langle L, M \rangle$ in the coordinates $u^1,\dots, u^n$, we obtain:
\begin{equation}\label{hold1}
\begin{aligned}
    \pd{M^i_k}{u^q} L^q_j  +& M^i_q \pd{L^q_j}{u^k} - \pd{L^i_j}{u^q} M^q_k  - L^i_q \pd{M^q_k}{u^j}   = \pd{M^i_q}{u^k} L^q_j  + M^i_q \pd{L^q_j}{u^k} - \pd{L^i_q}{u^j} M^q_k - L^i_q \pd{M^q_k}{u^j}  \\
    = &   \pd{}{u^k} \Big( M^i_q L^q_j\Big)- \pd{}{u^j} \Big( L^i_q M^q_k\Big)  =  \pd{}{u^k} \Big( M^i_q L^q_j\Big) - \pd{}{u^j} \Big( M^i_q L^q_k\Big). 
\end{aligned}    
\end{equation}
The last equation is due to the commutativity of $L$ and $M$. The symmetric part of $\langle L, M \rangle$ takes the form
\begin{equation}\label{hold2}
\begin{aligned}
\pd{}{u^j} \Big( M^i_q L^q_k\Big) - \pd{}{u^k} \Big( M^i_q L^q_j\Big) + \pd{}{u^k} \Big( M^i_q L^q_j\Big) - \pd{}{u^j} \Big( M^i_q L^q_k\Big) = 0. 
\end{aligned}    
\end{equation}
Equation \eqref{hold2} implies the first statement. If, in addition, all $\ddd u^i$ are conservation laws for $ML$, then the r.h.s. of  \eqref{hold1} vanishes, i.e., $\langle L, M\rangle=0$, which proves the second statement.
\end{proof}

In the next lemma, we assume that we work in either the  analytic category or the  formal category. Consider the system of quasilinear PDEs related to the operators $K_1,\dots, K_n$:
\begin{equation}\label{sys}
u_{t_j} = K_j u_{x}, \quad j = 1, \dots, n    
\end{equation}

\begin{Lemma}\label{dfg4}
For any  solution $u^i = u^i(x; t_1,\dots, t_n)$ of \eqref{sys} such that $\pd{u}{t}$ is invertible, the operator field
$$
K = t_1(u;x) K_1 + \dots + t_n(u;x) K_n
$$
is a (formal) common symmetry of $K_j$'s for any fixed value of $x$.  Here $t_i(u;x)= t_i(u^1,\dots, u^n, x)$ are obtained by inverting the map $(t_1,\dots,t_n) \mapsto (u^1,\dots,u^n)$ thinking of $x$ as a parameter.
\end{Lemma}

\begin{proof}
In the formal category, the condition that $K_j$ are symmetries of one another implies the existence of a formal solution for any initial formal curve $u_0(x)=u(x; 0, \dots, 0)$. Taking $\frac{\ddd u_0}{\ddd x} = \xi$ to be generic in the sense that $K_1\xi,\dots,K_n\xi$ are linearly independent, we guarantee that $\pd{u}{t}$ is invertible. In the analytic category, one applies the Cauchy-Kovalevskaya theorem to obtain the analytic solutions. 

The system \eqref{sys} implies that
$$
(K_j)^i_s \pd{u^s}{t_m} = (K_j)^i_s (K_m)^s_r \pd{u^r}{x} = a_{jm}^s (K_s)^i_r \pd{u^r}{x} = a_{jm}^s \pd{u^i}{t_s}.
$$
Denoting $(A_j)^s_m = a^s_{jm}$, we get
$$
K_j \pd{u}{t} = \pd{u}{t} A_j.
$$
or, equivalently, using $\pd{t}{u} = \left(\pd{u}{t}\right)^{-1}$:
$$
\pd{t}{u} K_j =  A_j \pd{t}{u}.
$$
Note that the last identity is exactly relation \eqref{epic1} from Lemma \ref{dfg1}  for $K=t_1 K_1 + \dots + t_nK_n$.  Hence, by Lemma \ref{dfg1},  $K$ is a common symmetry of $K_j$'s, $j=1,\dots,n$.
\end{proof}

In the formal and analytic categories, the second statement of Theorem \ref{t2} and  Lemma \ref{dfg4}  imply that the function $\sum a_i t_i$ is a common conservation law of $M^j$'s. Taking different initial conditions, one may construct $n$ conservation laws of this kind, such that their differentials are linearly independent. By Lemma \ref{dfg3}, this guarantees that $M^j$'s are symmetries of each other.

In the smooth category, we take an arbitrary point $x\in\mathsf M$ and pass to the formal category, which involves replacing all smooth objects with their Taylor expansions at this point. One will find, in this setup, that at the point $x$,   the operators  $M^i$ and $M^j$ are formal symmetries of each other, that is, the symmetric part of $\langle M^i, M^j \rangle$ vanishes at $x$ together with all of its derivatives.  Since $x\in \mathsf M$ is arbitrary,  the Taylor expansion of this tensor field is trivial at each point, which implies that the symmetric part of $\langle M^i, M^j \rangle$ vanishes identically. In other words,  $M^i$ and $M^j$ are symmetries of each other not only {\it formally}, but also in the usual (smooth) sense.  This completes the proof of Theorem \ref{t2}.
\end{proof}

 Notice that the proof of Theorem \ref{t2} implies the following statement. 

\begin{Corollary} In the real analytic case, 
 any operator Frobenius algebra containing the identity operator and generated by mutual symmetries admits $n$ independent conservation laws.
\end{Corollary}


\section{Nijenhuis operators and their symmetry algebras}\label{s3}

Here we are dealing with a special collection of commuting operator fields studied in Section \ref{s2}. For the purposes of the next section, it will be convenient to denote them by $M^1, \dots , M^n$  (in Section \ref{s2}, the initial basis was $K_1,\dots, K_n$, but due to duality there is no essential difference between $K_i$'s and $M^j$'s). 

Thus, we consider $n$ operator fields  $M^1, \dots, M^n$ on $\mathsf M$, $\dim \mathsf M=n$, which are symmetries of each other and satisfy properties (A1) and (A2) from Theorem \ref{t1}. In addition, we assume that all operators $M$ from the operator Frobenius algebra $\mathcal M= \Span_{\R}(M_1,\dots, M_n)$ are Nijenhuis, that is, the Nijenhuis torsion 
$\mathcal N_M \overset{\mathrm{def}}{=} \langle M, M\rangle$ vanishes. Equivalently, we may require that $\langle M_i, M_j\rangle =0$, $i,j=1,\dots,n$, i.e., $M_i$ are {\it strong} symmetries of each other.

We define the symmetry algebra $\Sym$ of $M^i$'s (or equivalently, of the operator Frobenius algebra $\mathcal M$) as the vector space of common {\it strong symmetries} of $M^1, \dots, M^n$:
\begin{equation}
\label{eq:n12}
\Sym=\Sym(M^1,\dots, M^n) = (\bigl\{ M  ~|~ M {\cdot}M_i = M_i {\cdot}M, \ \   \langle M,   M_i\rangle =0,  \ i=1,\dots,n\bigr\}.
\end{equation}

Recall   that the set of strong symmetries of a given operator is closed under usual matrix multiplication (Lemma 2.1 in \cite{st}) so that $\Sym$ is an (unital associative infinite-dimensional) algebra. 
The next theorem describes basic properties of the algebra $\Sym$.

\begin{Theorem}\label{t3}
In the setup above, the following holds:
\begin{enumerate}
    \item  Every operator $M\in \Sym$ can be uniquely written as $M = g_1 M^1 + \dots + g_n M^n$ for some functions $g_i, i = 1, \dots, n$. In particular, the algebra $\Sym$ is commutative.  Moreover, at each point $x\in\mathsf M$,  the subspace $\Sym_x = \{ M(x)~|~ M\in\Sym\} \subset \gl(T_x\mathsf M)$ is a Frobenius subalgebra.
    \item Every common symmetry $M$ of the operators $M^i$'s is necessarily strong and therefore belongs to $\Sym$.
 All $M, N \in \Sym$ are strong symmetries of each other.
    \item For each common conservation law $\alpha$  of $M_i$'s and each operator $M \in \Sym$, the $1$-form $M^* \alpha$ is closed. In other words, common conservation laws of $M^i$'s are common conservation laws for all elements of $\Sym$;
    \item In the formal and analytic categories,  there exists a common conservation law $\alpha$ of $M_i$'s such that ${M^i}^* \alpha$, $i=1,\dots,n$, are linearly independent.
\end{enumerate}
\end{Theorem}

\begin{Remark}\label{rem1}
\rm{
Theorem \ref{t3} implies that the operators $M^i$ in the definition of $\Sym$ can be replaced by any other operators $\widetilde M^i\in\Sym$ provided   
$\Span \bigl(M^1(x),\dots, M^n(x)\bigr)  = \Span \bigl(\widetilde M^1(x),\dots, \widetilde M^n(x)\bigr)$ at each point $x\in\mathsf M$. Hence, a particular choice of $M^i$'s can be ignored. For our purposes,  the algebra $\Sym$ itself is more important, and that is why we use the notation  $\Sym$ instead of $\Sym(M^1,\dots,M^n)$.
}    
\end{Remark}

\begin{proof}
The commuting operators $M^i$  satisfy conditions (A1) and (A2) from Theorem \ref{t1}. Since by definition,  $M$ commutes with $M^i$, $i=1,\dots, n$, then Lemma \ref{lem:2.1} implies that $M$ is a linear combination of $M^i$ whose coefficients may depend on $x\in \mathsf M$.  In other words,  $\Sym_x = \Span \bigl(M^1(x),\dots, M^n(x)\bigr)$.   Hence, the first statement follows from Theorem \ref{t1}. 

Now, let us proceed with the second one. As before, we define $a^{ij}_s$ as
$$
M^i {\cdot} M^j = a^{ij}_s M^s.
$$
Assume that $M$ and $N$ are common symmetries of $M^i$, $i=1,\dots, n$ (not necessarily strong).  Again by Lemma \ref{lem:2.1}, they have the form $M = g_1 M^1 + \dots + g_n M^n$ and $N = h_1 M^1 + \dots + h_n M^n$.  Let us show that they are strong symmetries of each other. This will prove  both parts of the second statement.  According to Lemma \ref{dfg1} and formula \eqref{epic1} we obtain
\begin{equation}\label{compute0}
{M^i}^* \ddd g_j = a^{is}_j \ddd g_s \quad \text{and} \quad {M^i}^* \ddd h_j = a^{is}_j \ddd h_s.    
\end{equation}
By direct computation, using the general identity
$$
\langle gA, hB\rangle = gh \langle A, B\rangle + hAB \otimes \ddd g - h A \otimes B^*\ddd g
+ gA^*  \ddd h \otimes  B - g\ddd h \otimes  AB   
$$
for the bracket $\langle~,~\rangle$, we get  
$$
\begin{aligned}
& \langle M, N \rangle  = \langle g_s M^s,  h_m M^m\rangle =   \\
=\ & g_s h_m \langle M^s, M^m\rangle 
+ h_mM^sM^m {\otimes} \,\ddd g_s 
- h_m M^s {\otimes} {M^m}^*\ddd g_s
+ g_s{M^s}^*  \ddd h_m {\otimes}  M^m 
- g_s\ddd h_m {\otimes} M^sM^m \\
=\ & g_s h_m \langle M^s, M^m\rangle   
+ h_m a^{ms}_r M^r {\otimes} \, \ddd g_s  
- M^r {\otimes} \, h_s {M^s}^* \ddd g_r   
+ g_s {M^s}^* \ddd h_r {\otimes} M^r    
- g_s  \ddd h_m {\otimes}\, a^{sm}_r M^r  \\
=\ &
g_s h_m \langle M^s, M^m\rangle  - 
 M^r \otimes  h_s \bigl({M^s}^* \ddd g_r - a^{sm}_r \ddd g_m\bigr) + g_s  \bigl( {M^s}^* \ddd h_r -  a^{sm}_r \ddd h_m \bigr)\otimes M^r  = 0, 
\end{aligned}   
$$
which completes the proof of the second statement.
Here, at the final step, we used \eqref{compute0} and the fact that $\langle M^s, M^m\rangle =0$.

For the third statement, take $\alpha =\ddd f$ to be a common conservation law of $M^i$'s. This implies that $\ddd (M^{i*} \ddd f) = 0$ for $i = 1, \dots, n$.  Hence, by direct computation, we obtain
$$
\ddd (M^* \ddd f) = \ddd(g_s {M^s}^* \ddd f) = \ddd g_s \wedge {M^s}^* \ddd f.
$$
Let us calculate the value of this 2-form on the pair of vector fields $M^i\xi, M^j \xi$, where $\xi$ is generic in the sense that ${M^1}\xi, \dots, {M^n}\xi$ are linearly independent, i.e., form a basis of the tangent  space.
 In the formula below, we use  $\langle \,;  \rangle$ for pairing of 1-forms and vector fields:
$$
\begin{aligned}
\ddd g_s \wedge {M^s}^* \ddd f (M^i\xi, M^j\xi)   &= \langle  \ddd g_s; M^i \xi \rangle \langle {M^s}^* \ddd f; M^j \xi \rangle - \langle \ddd g_s; M^j \xi \rangle \langle {M^s}^* \ddd f; M^i \xi \rangle = \\
& = \langle  \ddd g_s; M^i \xi \rangle \langle \ddd f; M^s M^j \xi \rangle 
      - \langle \ddd g_s; M^j \xi \rangle \langle \ddd f; M^s M^i \xi \rangle \\
& = a^{js}_m \langle  \ddd g_s; M^i \xi \rangle \langle \ddd f; M^m \xi \rangle  
      - a^{is}_m \langle  \ddd g_s; M^j \xi \rangle \langle \ddd f; M^m \xi \rangle \\ 
& = \langle  a^{js}_m \ddd g_s; M^i \xi \rangle \langle \ddd f; M^m \xi \rangle  
      - \langle  a^{is}_m \ddd g_s; M^j \xi \rangle \langle \ddd f; M^m \xi \rangle \\
& = \langle  M^{j*} \ddd g_m; M^i \xi \rangle \langle \ddd f; M^m \xi \rangle 
      - \langle  M^{i*} \ddd g_m; M^j \xi \rangle \langle \ddd f; M^m \xi \rangle \\
& = \langle  \ddd g_m; M^i M^j \xi \rangle \langle \ddd f; M^m \xi \rangle  
      - \langle  \ddd g_m; M^i M^j \xi \rangle \langle \ddd f; M^m \xi \rangle =0.
\end{aligned}
$$
Here we used identity  $M^{j*} \ddd g_m =   a^{js}_m \ddd g_s$   from Lemma \ref{dfg1}.  As ${M^1}\xi, \dots, {M^n}\xi$ form a basis of the tangent space, we conclude that the corresponding 2-form identically vanishes. Thus, $\ddd (M^* \ddd f) = 0$, and the third statement is proved.

To prove the fourth statement, first notice that in view of already proven statements, it is sufficient to prove it for any particular basis in $\Sym$  (see also Remark \ref{rem1}).  Hence, without loss of generality,  we may assume that $M^1 = \Id$. In particular, $M^i$'s satisfy  \eqref{eq:Id_assumption} and all the statements of Theorem \ref{t2} hold with $M^i$ and $K^j$ interchanged.
Thus, we get a collection of operators $K_i$, which are symmetries of each other.

Due to duality between $M^i$ and $K_j$,  Theorem \ref{t2}  and its proof still work in their original form, i.e., without interchanging $M^i$ and $K_j$. In particular, in the formal or analytic category, Lemma \ref{dfg4} implies the existence of a symmetry $K = t_1(u, x) K_1 + \dots + t_n(u, x) K_n$ whose coefficients (due to the second statement of Theorem \ref{t2}) provide a common conservation law for $M^i$ with the required properties. This completes the proof  of Theorem \ref{t3}.
 \end{proof}

As explained in Remark \ref{rem1}, the {\it generating} operators $M^i\in\Sym$ (we may refer to them as a {\it basis})  can be chosen in many different ways. The only condition is that every $M\in \Sym$ is a linear combination of $M_i\in\Sym$ with functional coefficients. 
We say that $\Sym$ is {\it flat} if there exists a basis $M^1,\dots, M^n$ of $\Sym$ such that, in some coordinate system,  the matrices of $M^i$ are all constant. We call both the corresponding basis and coordinate system {\it flat}. The flat coordinates are in no way unique.

Locally, flat symmetry algebras $\Sym$ are in natural correspondence with Frobenius algebras $(\mathfrak a,\star)$. Indeed, take an arbitrary Frobenius algebra $\mathfrak a$ with a basis $e_1,\dots, e_n$ and consider the matrices $M^i\in \gl(n,\R)$ of the operators $\mathsf R_{e_i}$ in the basis $e_1,\dots, e_n$ (recall that $\mathsf R_\xi: \mathfrak a \to \mathfrak a$ is defined by $\mathsf R_{\xi}\eta = \eta \star \xi$).
We may think of $M^i$'s as constant commuting  operator fields on $\mathsf M = \R^n$. By Theorem \ref{t1}, they satisfy conditions (A1) and (A2).  Then, by construction, the symmetry algebra $\Sym =\Sym(M^1,\dots, M^n)$  is flat and at each point, $\Sym_{x}$ is obviously isomorphic to $\mathfrak a$.  Conversely, if $M^1,\dots, M^n$ form  a flat basis of $\Sym$, then using flat coordinates, we may consider $M^i$'s as constant $n\times n$ matrices such that $\Span(M^1,\dots, M^n)\subset \gl(n,\R)$ is a Frobenius algebra (more specifically, the image of the regular representation of a certain Frobenius algebra $\mathfrak a$). The next theorem describes local structure of flat symmetry algebras $\Sym$ constructed in this way.

\begin{Theorem}\label{t4}
Let $\Sym$ be a flat symmetry algebra with a flat basis $M^1,\dots, M^n$ (we may think of them as constant $n\times n$ matrices).  Consider a constant vector field $\xi$ such that $M^1\xi, \dots, M^n\xi$ are linearly independent\footnote{Here we mean that $\xi$ has constant components in the flat   coordinate system associated with $M^i$'s. The existence of such a vector field is guaranteed by condition (A1).} and introduce flat coordinates 
$u^1, \dots, u^n$ such that $\pd{}{u^i} = M^i \xi$. Then, 
\begin{enumerate}
    \item The operator field 
    $$
    U = u^1 M^1 + \dots + u^n M^n
    $$
    is a common strong symmetry of $M^i$'s, i.e., $U \in \Sym$.
    \item For any collection of analytic functions $f_i$,  $i = 1, \dots, n$, the operator
    \begin{equation}\label{solve}
     M = f_1(U) M^1 + \dots + f_n(U) M^n,  
    \end{equation}
     is a common strong symmetry of $M^i$'s, i.e., $M\in \Sym$.
    \item Locally, in the analytic category  every operator $M\in \Sym$  can be written in the form \eqref{solve} for an appropriate choice of $f_i$, $i=1,\dots,n$.
\end{enumerate}

\end{Theorem}

\begin{Remark}{\rm
Formula  \eqref{solve} describes all real analytic common symmetries of $M^i$, i.e., gives an explicit local parametrisation of $\Sym$ in terms of $n$ functions of one variable in the analytic category.  There can, however, exist smooth symmetries which are not real-analytic. The existence and description of such symmetries essentially depend on the Frobenius algebra $\mathfrak a = \Span (M^1,\dots,M^n)$. Roughly speaking, formula \eqref{solve} still works for smooth functions $f_i$ as soon as $f_i(U)$ is well defined,  but this is not always the case.  
}
\end{Remark}

\begin{proof}
As before, we introduce $a^{ij}_s$ from $M^i{\cdot}M^j = a^{ij}_s M^s$. Then, in the coordinates $u^1,\dots, u^n$, we have
$$
M^i \pd{}{u^j} = M^i M^j \xi = a^{ij}_s M^s \xi = \sum_{s=1}^n a^{ij}_s \pd{}{u^s},
$$
which also implies  
\begin{equation}
\label{eq:n14}
{M^i}^* \ddd u^j = \sum_{s = 1}^n a^{is}_j \ddd u^s. 
\end{equation}
Replacing $K$ with $U = u^1 M^1 + \dots + u^n M^n$ and $K_i$ with $M^i$ in formula \eqref{top4}, we get   
$$
\langle M^i, U \rangle  
= \sum_{s=1}^n u^s \langle M^i, M^s \rangle + \sum_{j=1}^n \left({M^i}^* \ddd u^j -  \sum_{s=1}^n a^{is}_j\ddd u^s \right) \otimes M^j  
$$
The first summand vanishes since $M^i$ are constant in the coordinates $u^1,\dots, u^n$,  and the second summand vanishes because of  \eqref{eq:n14}. Thus, $\langle M^i, U \rangle = 0$, meaning that $U$ is a common strong symmetry of $M^i$'s and, therefore, $U\in\Sym$.

Recall that the product of strong symmetries is a strong symmetry (Lemma 2.1 in \cite{st}). Therefore,  every polynomial $p(U)$ (with constant coefficients) or matrix analytic function $f(U)$ is a strong symmetry also.  Hence, formula \eqref{solve} indeed produces a common strong symmetry for $M^i$. Thus, the second statement holds.

Now, let us proceed with the third statement. Every common symmetry $M$ has the form $M = g_1(u) M^1 + \dots + g_n(u) M^n$,  where $g_i$'s satisfy relations \eqref{compute0}. Due to 
\eqref{eq:n14}, we notice that $(M^i)^j_k = a^{ik}_j$. Therefore, \eqref{compute0} can be written as the following over-determined system
$$
M^i \left( \pd{g}{u}\right) = \left( \pd{g}{u}\right) M^i, \quad i = 1, \dots, n,
$$
where $\Bigl( \pd{g}{u} \Bigr)$ is the Jacobian  matrix for the system of functions $g_1, \dots, g_n$. Due to Lemma \ref{lem:2.1} the matrix $\Bigl(\pd{g}{u} \Bigr)$ belongs to $\Span_{\R}(M^1, \dots, M^n)$. 

This implies that $n(n - 1)$ derivatives from $\pd{g}{u}$ can be expressed as linear combinations of the remaining  $n$ derivatives. More specifically,  using the fact that $M^i \xi = \frac{\partial}{\partial u^i}$,  we can observe that every local solution  $g=(g_1, \dots, g_n)$ is uniquely determined by its restriction to the straight line $u(t) = t\xi$ or, in more detail, $u^\alpha(t) = t \xi^\alpha$, where $\xi^\alpha$ are the components of $\xi$ in the coordinates $u^1,\dots, u^n$.  The next observation is that $\xi^1 M_1 +\dots + \xi^n M^n = \Id$.  This easily follows from the fact that two operators $M, M'\in \Span (M^1,\dots,M^n)$ coincide if and only if $M\xi = M'\xi$  (see assumption (A1)): indeed, we have $(\xi^1 M_1 +\dots + \xi^n M^n)\xi = \sum_i \xi^i \frac{\partial}{\partial u^i} = \xi = \Id \, \xi$.  

Now consider an arbitrary symmetry  $g_1 (u) M^1 + \dots g_n(u) M^n$ with initial conditions $g_i(u(t)) = f_i(t)$ and $f_i$ being  analytic functions.   We claim that this symmetry can be written in the form of \eqref{solve}, namely,
$M = f_1(U) M^1 + \dots + f_n(U)M^n$. To prove this, it suffices to notice that the restriction of $U$ onto the curve $u(t)$ is
$$
U(u(t)) = u^1(t) M^1 + \dots + u^n(t) M^n = t ( \xi^1 M^1 + \dots + \xi^n M^n ) = t \, \Id,
$$
and, therefore, the restrictions of $M=f_1(U) M^1 + \dots + f_n(U)M^n$ and  $g_1 (u) M^1 + \dots g_n(u) M^n$ onto the initial curve $u(t)$ coincide. Indeed,   
$$
M(u(t)) = \sum_i f_i (t\,\Id) M^i = \sum_i f_i(t) M^i  =  \sum g_i(u(t)) M_i,
$$
which completes the proof.
\end{proof}

\begin{Remark}
\rm{
A natural and important example of algebras $\Sym$ is the symmetry algebra $\Sym\,L$ of a $\gl$-regular Nijenhuis operator $L$, that is, 
$$
\Sym\, L = \{ M\, \mbox{such that } ML=LM \mbox{ and } \langle M,L\rangle (\xi,\xi) = 0 \  \forall \xi  \}
$$
In this case, all symmetries $M\in \Sym\, L$ are strong, i.e., $\langle M,L\rangle=0$ (see Theorem 1.2 in \cite{nij4}). The powers of $L$, i.e., $\Id, L,\dots, L^{n-1}$ form a natural basis of $\Sym\, L$. If $L$ is algebraically generic (i.e., the multiplicities of its eigenvalues do not change), then the results of \cite{nij4} (Theorems 1.1 and 1.4) imply that 
$\Sym$ is locally flat.
}    
\end{Remark}


\section{Application of symmetry algebras to finite-dimensional integrable systems}\label{s4}

In this section, we study Poisson commuting functions $F_s: T^*\mathsf M \to \R$, $s = 1, \dots, n$,  quadratic in momenta. They can be naturally identified with symmetric bilinear forms $h_s$ on $T^*\mathsf M^n$:
$$
F_s(x,p) = \sum_{i,j} h_s^{ij}(x) p_i p_j = h_s (p,p). 
$$
If $h_1$ is non-degenerate,  then we can interpret it as a contravariant (pseudo-)Riemannian  metric, and the functions can be equivalently written in the form
$
F_s  =  h_1 ( K_s^* p , p),
$
where the operators 
\begin{equation}\label{killing0}
K_s = h_s h_1^{-1}.     
\end{equation}
are Killing $(1, 1)$-tensors of the metric $h_1$. 
The next theorem provides a natural method for constructing such functions from a symmetry algebra $\Sym$.

\begin{Theorem}\label{t5}
Let $\Sym$ be a symmetry algebra,  $M^1, \dots, M^n$ be a basis of    $\Sym$, and $\alpha$ be a common conservation law of $M^i$'s such that ${M^1}^*\alpha, \dots, {M^n}^*\alpha$ are linearly independent pointwise. Consider the canonical coordinate system $u^1, \dots, u^n, p_1,\dots,p_n$ on $T^*\mathsf M$ such that $\ddd u^i = M^{i*} \alpha$, $i = 1, \dots, n$ and set, as before,  $M^i {\cdot} M^j = a^{ij}_s M^s$, $a^{ij}_s=a^{ij}_s(u)$. Then
\begin{enumerate}  
    \item The quadratic in momenta functions
    $$
    F_s(u,p) = a^{ij}_s(u) p_i  p_j, \quad s = 1, \dots, n,
    $$
    pairwise Poisson commute with respect to the canonical Poisson structure on $T^*\mathsf M$;
    \item The functions $F_s(u,p)$ satisfy the condition
    $$
    \det \Big( \pd{F}{p}\Big) \neq 0\qquad \mbox{for a generic $p\in T^*_x\mathsf M$};
    $$
    \item $M^i$ are self-adjoint with respect to the forms $h_s=\left( a^{ij}_s \right)$, $s = 1, \dots, n$. If $h_1 = \left(a^{ij}_1\right)$ is non-degenerate, then the corresponding Killing tensors $K_s$ given by  \eqref{killing0} are dual to $M^i$  with respect to $\left(a^{ij}_1\right)$ treated as a Frobenius form in the sense of Theorem \ref{t2}\footnote{Notice that $a^{ij}_1 = a^{ij}_k a^k$, where $a=(a_1,\dots, a_n)=(1,0,\dots, 0)$, which agrees with the assumptions of Theorem \ref{t2}.}, that is, $M_i$ and $K_s$ are related by
    $$
    M^i = a^{is}_1 K_s.   
    $$    
\end{enumerate}
\end{Theorem}

\begin{proof}
We start with the second statement which is purely algebraic. 
We  have $\pd{F_s}{p_j}  = 2 a^{ij}_s p_i$.   To show that for a generic $p\in T^*_x\mathsf M$, the matrix   $\left(\pd{F_s}{p_j} \right)$ is non-degenerate,  we recall that at each point $x\in \mathsf M$, the space $\Span (M^1(x),\dots, M^n(x))$ contains the identity operator, i.e.,  $\Id = \sum a_i M^i(x)$ for some $a_i\in \R$. This implies that $a^{ij}_s a_i = \delta^j_s$ so that generically the matrix  $\left(a^{ij}_s p_i\right)$  is non-degenerate.
Thus, the second statement  holds.

Notice that the functions $F_s$ satisfy the (matrix) relation
\begin{equation}
\label{eq:n15}
F_1 M^1 + \dots + F_n M^n = \Bigl(p_1 M^1 + \dots + p_n M^n\Bigr)^2,
\end{equation}
and can be defined from it. Indeed,
$$
\Bigl(p_1 M^1 + \dots + p_n M^n\Bigr)^2 = \sum_{i,j} M^i{\cdot}M^j p_i p_j =
\sum_{i,j,s}  a^{ij}_s M^s p_i p_j = \sum_s \left( \sum_{i,j} a^{ij}_s p_i p_j\right) M^s = F_s M^s. 
$$

Since $\det \left(\pd{F_s}{p_j}\right)\ne 0$ for a generic point $(x_0,p_0)\in T^*\mathsf M^n$, then in a neighbourhood of such a point, the relations $F_s(x,p) = c_s$, $s=1,\dots, n$, can be resolved with respect to $p$, i.e., there exist smooth functions $g_i(u;c)$ such that the graph $p_i = g_i(u; c)$ coincides with the common level surface $\mathsf X_c=\{F_s(u,p) = c_s, \ s=1,\dots,n\}$ of the functions $F_s$ (here we treat $c_i$ as parameters). Therefore, in view of \eqref{eq:n15},  the operator field $M=g_1 M^1 + \dots + g_n M^n$ satisfies the identity 
$$
M^2 = S, \quad \mbox{where $S\overset{\mathrm{def}}{=} c_1 M^1 + \dots + c_n M^n$}.
$$

Without loss of generality we may assume that $\det M_0 \ne 0$, where $M_0(u) = p_1^0 M^1(u)+\dots+ p_n^0 M^n(u)$.  Let us set $M = \widehat M {\cdot} M_0$.  Notice that $M$, $M_0$, $\widehat M$ and $S$ commute.  Hence,  the relation $M^2 = S$ can be written as
$\widehat M^2 = \widehat S$, where $\widehat S = S M_0^{-2}$.    By construction, at the point $u_0$ we have $\widehat M (u_0) = \Id$.   We now use the fact that in a neighborhood of the identity $\Id \in \gl(n,\R)$,  the matrix map $\widehat M \to \widehat M^2 = \widehat S$ is invertible, and moreover, the inverse map 
$\widehat M =  \sqrt{\widehat S}$ is a real analytic matrix function (given by the same Taylor series as the usual square root, i.e., $\sqrt{1+x} = 1 +\frac{1}{2} x - \frac{1}{8} x^2 + \dots$).  Thus we can write $M$ in the form
$$
M = M_0 \sqrt{S M_0^{-2}}
$$

Notice that for each $c=(c_1,\dots,c_n)$,  the operator $S$ belongs to $\Sym$.  The same is true for $M_0$ and $M_0^{-2}$, since $M_0$ is a linear combination of $M^i$'s with constant coefficients. Thus, $M\in \Sym$ too. Therefore, by Theorem \ref{t3}, $\alpha$ is a conservation law for $M$, i.e.,  $M^*\alpha$ is closed. On the other hand,
$$
M^* \alpha = g_i {M^i}^* \alpha = g_i \, \ddd u^i.
$$
Thus, the closedness of $M^*\alpha$ implies $\pd{g_i}{u^j} = \pd{g_j}{u^i}$ for all $i,j=1,\dots,n$. This means that $X_c$ is Lagrangian for generic $c$ and, therefore, $h_i$'s Poisson commute. Thus, the first statement is proved.

To show that $M^i$'s are self-adjoint w.r.t. $h_s$, we need to check the identity 
\begin{equation}
\label{eq:n21}
(M^i)^j_k  (h_s)^{km} = (M^i)^m_k  (h_s)^{kj}   
\end{equation}
where $(M^i)^j_k$ and $(h_s)^{km}$ denote the components of $M^i$ and $h_s$ in coordinates $u^1,\dots, u^n$.  By construction, $(h_s)^{km}= a_s^{km}$ and 
$
{M^i}^* \ddd u^j = {M^i}^* {M^j}^* \alpha = a^{ij}_s {M^s}^* \alpha = a^{ij}_s \ddd u^s
$,
which means that 
$(M^i)^j_k = a^{ij}_k = a^{ji}_k$. Hence \eqref{eq:n21} takes the form
$$
a^{ij}_k a_s^{km} = a^{im}_k  a_s^{kj},   
$$
which is true, since $a^{ij}_m$ is the structure tensor of a commutative associative algebra.

Finally, assume that the form $b^{ij} = a^{ij}_1$ is non-degenerate.  Then, in the context of Theorems \ref{t1} and \ref{t2}, we can think of it as a Frobenius form related to the operators $M^1,\dots, M^n$  (the components of the corresponding covector $a$ in the  basis $M^1,\dots, M^n$ are $a=(1,0,\dots,0)$).  The Killing tensors are given by \eqref{killing0} or, equivalently,  in our coordinates $u^1,\dots,u^n$: 
$$
(K_s)^j_k = b_{km} a_s^{mj},
$$ 
where $b_{km}$ denotes the inverse of $b^{ij}$.  Then,
$$
a^{is}_1 (K_s)^j_k = a^{is}_1\, b_{km} \, a_s^{mj} =  a^{ms}_1\, b_{km}\, a_s^{ij} = \delta^s_k \, a_s^{ij} = a_k^{ij} = (M^i)^j_k,
$$
or shortly, $a^{is}_1 (K_s) = M^i$, as required.

\end{proof}

\begin{Remark}
\rm{
From the argument in the proof above if all the real eigenvalues of $c_1 M^{1*} + \dots + c_n M^{n*}$ are positive, the differential of Hamilton's principal function is given by the formula
$$
\ddd W(u, c) = \sqrt{c_1 M^{1*} + \dots + c_n M^{n*}} \, \alpha,
$$
where the root is understood in the sense discussed above.
}    
\end{Remark}

The next example yields an integrable system for the flat symmetry algebra $\Sym$, corresponding to the Frobenius algebra  \eqref{eq:n16} from Example \ref{ex:3.2}.

\begin{Ex}
\rm{
The natural basis of the Frobenius algebra \eqref{eq:n16} is formed by four matrices
$$
M_1 = \begin{pmatrix}
1 & 0 & 0 & 0\\
0 & 1 & 0 & 0 \\
0 & 0 & 1 & 0 \\
0 & 0 & 0 & 1 \\
\end{pmatrix}, \quad
M_2 = \begin{pmatrix}
0 & 0 & 0 & 0\\
1 & 0 & 0 & 0 \\
0 & 0 & 0 & 0 \\
0 & 1 & 0 & 0 \\
\end{pmatrix}, \quad 
M_3 = \begin{pmatrix}
0 & 0 & 0 & 0\\
0 & 0 & 0 & 0 \\
1 & 0 & 0 & 0 \\
0 & 0 & 1 & 0 \\
\end{pmatrix}, \quad
M_4 {=} \begin{pmatrix}
0 & 0 & 0 & 0\\
0 & 0 & 0 & 0 \\
0 & 0 & 0 & 0 \\
1 & 0 & 0 & 0 \\
\end{pmatrix},
$$
which we treat as (constant) operator fields on $\R^4$.  The common strong symmetries of $M_i$'s, i.e., elements of $\Sym$ have the following form
$$
\begin{pmatrix}
f_1  & 0 & 0 & 0 
\\
u_2 \,  f_1' +f_{2} & f_{1}  & 0 & 0 
\\
 u_3 \, f_1' +f_{3} & 0 & f_{1} & 0 
\\
 \frac{\left(u_{2}^{2}+u_{3}^{2}\right) f_1''}{2}+ u_2  f_2' + u_3 f_3' +  u_4 f_1' + f_4
& u_2   f_1' +f_{2}  & u_3  f_1' +f_{3} & f_{1} 
\end{pmatrix},
$$
where $f_1, f_2, f_3$ and $f_4$ are arbitrary functions of $u_1$.
Here we used Theorem \ref{t4}.

If we take $\alpha=\ddd u_4$ as a common conservation law, then the Poisson commuting Hamiltonians corresponding to the basis $M_1, M_2, M_3, M_4 \in \Sym$  are as follows
$$ \{F_{1} = 2 p_{1} p_{4}+p_{2}^{2}+p_{3}^{2}, F_{2} = 2 p_{3} p_{4}, 
F_{3} = 2 p_{2} p_{4}, F_{4} = p_{4}^{2}\}.
$$

They trivially commute with respect to the Poisson bracket. 

If we consider a different  (non-constant) basis of $\Sym$:
$$
\widetilde M_1 = 
\begin{pmatrix}
1 & 0 & 0 & 0\\
0 & 1 &  0 &  0 \\
0 & 0 & 1 & 0 \\
0 & 0 & 0 & 1 \\
\end{pmatrix} \quad \mbox{with } f_1  = 1,\  f_2=f_3=f_4=0,
$$
$$
\widetilde M_2 = \begin{pmatrix}
u_1 & 0 & 0 & 0\\
u_2 & u_1 & 0 & 0 \\
u_3 & 0 & u_1 & 0 \\
u_4 & u_2 & u_3 & u_1 \\
\end{pmatrix}\quad  \mbox{with }
f_1= u_1,\  f_{2}=f_3=f_4=0, 
$$
$$
\widetilde M_3 = \begin{pmatrix}
0 & 0 & 0 & 0\\
u_1 & 0 & 0 & 0 \\
0 & 0 & 0 & 0 \\
u_2 & u_1 & 0 & 0 \\
\end{pmatrix} \quad  \mbox{with }
f_1 = 0, \ f_2 = u_1, \ f_3=f_4=0,
$$
$$
\widetilde M_4 = \begin{pmatrix}
u_1^2 & 0 & 0 & 0\\
2u_1u_2 & u_1^2 & 0 & 0 \\
2u_1u_3 & 0 & u_1^2 & 0 \\
u^2_2{+}u_3^2 {+} 2u_1u_4 & 2u_1u_2 & 2u_1u_3 & u_1^2 \\
\end{pmatrix} \quad  \mbox{with }
f_1  = u_1^2, \  f_2=f_3=f_4=0,
$$
 in the domain $\{ u_1 u_3 (u_2^2 + u_3^2)\ne 0\}$, where these operators are pointwise linearly independent, then the commuting Hamiltonians will be quite nontrivial:   
\begin{eqnarray*}
F_{1} &=& 
\frac{2 p_{1} p_{4} u_{1}^{2} u_{3}+p_{2}^{2} u_{1}^{2} u_{3}-2 p_{2} p_{4} u_{1} u_{2} u_{3}+p_{3}^{2} u_{1}^{2} u_{3}-2 p_{3} p_{4} u_{1}^{2} u_{4}-2 p_{3} p_{4} u_{1} u_{3}^{2}+p_{4}^{2} u_{2}^{2} u_{3}+p_{4}^{2} u_{3}^{3}}{\left(u_{2}^{2}+u_{3}^{2}\right) u_{3}},
\\  
F_{2} &= &
-\frac{2 \left(2 p_{1} p_{4} u_{1} u_{3}+p_{2}^{2} u_{1} u_{3}-2 p_{2} p_{4} u_{2} u_{3}+p_{3}^{2} u_{1} u_{3}-2 p_{3} p_{4} u_{1} u_{4}+p_{3} p_{4} u_{2}^{2}-p_{3} p_{4} u_{3}^{2}\right)}{\left(u_{2}^{2}+u_{3}^{2}\right) u_{3}},\\
F_{3} & = &
\frac{2 p_{4} \left(p_{2} u_{3}-u_{2} p_{3}\right)}{u_{1} u_{3}}, \\
F_{4} & =& 
\frac{2 p_{1} p_{4} u_{1} u_{3}+p_{2}^{2} u_{1} u_{3}-2 p_{2} p_{4} u_{2} u_{3}+p_{3}^{2} u_{1} u_{3}-2 p_{3} p_{4} u_{1} u_{4}+2 p_{3} p_{4} u_{2}^{2}}{u_{1} u_{3} \left(u_{2}^{2}+u_{3}^{2}\right)}.
\end{eqnarray*}
}
\end{Ex}

The next theorem is the inverse of Theorem \ref{t5}.

\begin{Theorem}\label{t6}
Consider a collection of quadratic in momenta functions $F_s(x,p) = \sum_{i,j} h_s^{ij}(x) p_i p_j$, $s = 1, \dots, n$. Suppose that they satisfy the following conditions:
\begin{enumerate}
    \item[\rm(i)] Functions $F_s$ pairwise commute with respect to the canonical Poisson bracket on  $T^*\mathsf M$;
    \item[\rm(ii)]  $\det \Big( \pd{F}{p}\Big) \neq 0$ for generic $p$;
    \item[\rm(iii)] $h_1^{ij}$ is non-degenerate, and Killing tensors $K_i = h_i h_1^{-1}$ commute pairwise in the algebraic sense, that is,
    $K_i {\cdot} K_j = K_j {\cdot} K_i$. 
\end{enumerate}
Then the functions $F_s$ are obtained via the procedure described in Theorem \ref{t5} for an appropriate choice of a symmetry algebra $\Sym$, basis $M_1,\dots,M_n \in \Sym$, and a common conservation law $\alpha$.
\end{Theorem}
\begin{proof}
The idea of the proof is straightforward: we will recover all the necessary data, i.e., the Nijenhuis operators $M_i$ and the 1-form $\alpha =\ddd f$ from the quadratic forms $h_1, \dots, h_n$. 

Consider a common regular level surface of the quadratic functions $F_1,\dots, F_n$ 
$$
\mathsf X_{0} =\{ F_1(x,p)=c_1^0, \dots,  F_n(x,p)=c_n^0\} \subset T^*\mathsf M,
$$
such that the natural projection $\pi : \mathsf X_{0} \to \mathsf M$ is a (local) diffeomorphism. The existence of such a level surface $\mathsf X_0$ follows from condition (ii).

The Hamiltonian vector fields generated by $F_1, \dots, F_n$ are tangent to $\mathsf X_0$, linearly independent, and commute with respect to the standard Lie bracket. Let $\xi_1, \dots, \xi_n$ denote their natural projections on $\mathsf M$ by means of $\pi$.  Since $\pi$ is a diffeomorphism,  $\xi_1, \dots, \xi_n$ are still linearly independent and commute. Hence, we can choose local coordinates $u^1,\dots, u^n$ on $\mathsf M$ in such a way that $\xi_i = \pd{}{u^i}$, $i=1,\dots,n$. 

By construction, we have 
\begin{equation}\label{eq1}
K_i \partial_{u^1} = \partial_{u^i}.    
\end{equation}
In particular, this implies that $K_i$ are linearly independent. Thus, due to Lemma \ref{lem:2.1},   $\mathfrak A(x) = \Span (K_1(x), \dots, K_n(x))\subset \gl(T_x\mathsf M)$  is a unital commutative associative algebra. We denote $K_i K_j = a_{ij}^s K_s$. In our special  coordinates $u^1,\dots, u^n$ we have
\begin{equation}\label{killing}
K_i \pd{}{u^j} = K_i \left(K_j \pd{}{u^1}\right) = a^s_{ij} K_s \pd{}{u^1} = a^s_{ij}  \pd{}{u^s} \quad\mbox{and, in particular,}\quad   (K_i)_j^s = a^s_{ij}. 
\end{equation}
 
By construction, $K_i$ are self-adjoint with respect to $h^{-1}_1 = g$, that is,
$$
g_{is} a^s_{jm} = g_{ms} a^s_{ji}. 
$$
If we formally define the symmetric bilinear form $b$ on the algebra $\mathfrak A(x)$ by setting $b_{ij}\overset{\mathrm{def}}{=}b(K_i, K_j)=g_{ij}$, then the above relation amounts to
$$
b(K_j{\cdot}K_m, K_i) = b(K_m, K_j{\cdot} K_i), 
$$
which means that $b$ is a Frobenius form so that $\mathfrak A(x)$ is a Frobenius algebra.  In particular,  $K_1,\dots, K_n$ satisfy conditions (A1) and (A2) from Section \ref{s1}. 

From now on,  $p_1,\dots, p_n$ will denote the canonical momenta corresponding to the coordinates $u^1,\dots, u^n$. Since $\pi$ is a diffeomorphism,  $\mathsf X_0$ can be represented as a graph; that is,  
$$
p_1 = g_1 (u), \dots,  p_n= g_n(u),
$$ 
in other words, the relations $F_i (x,p)= c_i^0$ can be resolved with respect to $p$.   Consider the differential $1$-form
$$
\alpha = g_1(u) \ddd u^1 + \dots + g_n(u) \ddd u^n
$$ 
on $\mathsf M^n$, which is nothing else but the pull-back $\alpha = {\pi^{-1}}^* \bigl( p_i\,\ddd u^i\bigr)$ of the canonical action form from $\mathsf X_0 \subset T^*\mathsf M^n$ to $\mathsf M^n$. 

First we discuss purely algebraic properties of $K_i$ at a fixed tangent space $V=T_x\mathsf M$.   However, this point is arbitrary so that these properties will hold on every tangent space.

Consider the operator fields $M^1,\dots, M^n$ forming the basis in $\mathfrak A(x)$ dual to $K_1,\dots, K_n$ in the sense of the Frobenius form $b$.  We are in the situation studied in Section \ref{s1} so that all the formulas from this section work (we keep the same notation).  By \eqref{eq:n18},  we have $M^i = b^{ij} K_j$, $K_i=b_{ij} M^j$,  where  $b_{ij}=b(K_i, K_j)$ is the inverse of 
$b^{ij}= (h_1)^{ij}$.

\begin{Lemma}\label{l2}
The following formulas hold
\begin{equation}\label{form1}
   (M^j)^* \alpha = \tfrac{1}{2} \,\ddd u^j, \quad j = 1, \dots, n, 
    \end{equation}
    \begin{equation}\label{form2}
    M^i M^j = h^{ij}_s M^s,   
    \end{equation} 
\begin{equation}
\label{eq:b3}
(M^1 p_1 + \dots + M^n p_n)^2 = F_1 M^1 + \dots + F_n M^n.
\end{equation}
\end{Lemma}

\begin{proof} We start with formula \eqref{form2}, which describes the structure constants $a^{ij}_m$ of $\mathfrak A$ in the dual basis $M^1, \dots, M^n$.  We need to check that they coincide with 
$h^{ij}_m$ defined, in terms of the basis $K_1,\dots, K_n$, as follows
$h^{ij}_m =  b^{i s} (K_m)^j_s$. 
Using \eqref{eq:n17} and \eqref{killing}, we get
$$
a^{ij}_m \overset{\eqref{eq:n17}}{=} b^{i s} a_{sm}^j \overset{\eqref{killing}}{=} b^{i s} (K_m)^j_s = h^{ij}_m,
$$
as required.
We also notice that due to \eqref{killing} we have
\begin{equation}\label{eq:b2}
(M^k)^j_i = g^{k q} (K_i)^j_q = h^{k j}_i.    
\end{equation}

From \eqref{form2}  we immediately see that
$$
(p_1 M^1 + \dots + p_n M^n)^2 =  M^i M^j p_i p_j = h^{ij}_s p_i p_j M^s = 
F_1 M^1 + \dots + F_n M^n,
$$
which proves \eqref{eq:b3}.

Finally to prove \eqref{form1}, consider the projection of the Hamiltonian vector field of $F_i$ onto the base  $\mathsf M$.  The $k$-th coordinate of this projection is $\dot u^k = \frac{\partial F_i}{\partial p_k} = 2 (h_i)^{kj}p_j = 2(h_i)^{kj}g_j $, since  on the common level surface $\mathsf X_0$ we have $p_i = g_i$.   By construction, this projection coincides with the basis vector field 
$\partial_{u_i}$, which means that 
\begin{equation}
\label{eq:b1}
(h_i)^{kj}g_j =   \tfrac{1}{2}\delta^i_k. 
\end{equation}
Hence,
$$
\begin{aligned}
(M^j)^* \alpha & = b^{j s} K_s^*( g_i \ddd u^i )
 = g_i b^{j s} K_s^* \ddd u^i  = g_i b^{j s}  (K_s)^i_m \ddd u^m = \\ &= 
g_i b^{j s}  (K_m)^i_s \ddd u^m = g_i (h_m)^{j i} \ddd u^m =   \tfrac{1}{2}\delta_m^j \ddd u^m =   \tfrac{1}{2}\ddd u^j,
\end{aligned}
$$
as stated. 
\end{proof}
 
The second part of the proof is geometric. We will show that $M^i$ are Nijenhuis operators that are (strong) symmetries of      
each other and that $\alpha$ is a common conservation law for $M^i$.   

To that end, we use  Hamilton's principal function $W(u, c)$ associated with the Poisson commuting Hamiltonians $F_1,\dots, F_n$.  We briefly recall the definition of $W(u, c)$. We consider common level surfaces $\mathsf X_c = \{ F_1(u,p)=c_1, \dots, F_n(u,p)=c_n\}$  and represent each of them as a graph over $\mathsf M$:
$$
p_1 = g_1(u,c), \dots, p_n = g_n(u,c).
$$
In other words, we resolve the system of equations defining $\mathsf X_c$ with respect to $p$. Since $\mathsf X_c$ are all Lagrangian, the canonical action form  $p\,\ddd u = \sum g_k(u,c)\ddd u^k$  restricted to $\mathsf X_c$ is closed for each $c=(c_1,\dots,c_n)\in \R^n$, and therefore there (locally) exists a function $W(u,c)$ such that $\ddd W(x,c) = \sum g_k(u,c)\ddd u^k$.  Here $\ddd W$ denotes the differential of $W(u,c)$ with respect to the  variables $u=(u^1,\dots, u^n)$, while $c=(c_1,\dots, c_n)$ are treated as parameters. 
In particular, $\alpha = \ddd W(x, c^0)$ and $p_i = \frac{\partial W}{\partial u^i}$.

Formula \eqref{form1} implies that the covectors $(M^s)^* \alpha = (M^s)^* \ddd W(x, c^0)$ are linearly independent. Hence, the differential forms $(M^s)^* \ddd W(u,c)$ are linearly independent for all $c\in\R^n$ close to $c_0$. Then, by Remark \ref{p1}, there exists a unique operator field $\widetilde M = \widetilde M(u, c) \in \mathfrak A$ such that
    \begin{equation}\label{lap3}
       \widetilde M^* \alpha  =   \tfrac{1}{2}\ddd W \ \  \mbox{or, in more detail,} \ \ \widetilde M^*(u,c) \ddd W(u, c^0)  =   \tfrac{1}{2}\ddd W (u, c).
    \end{equation}
The construction of such a matrix is essentially solving a system of linear equations. This implies that it depends on coordinates and parameters in a good way: smooth or analytic, depending on the category which we are working with. The following Lemma holds. 

\begin{Lemma}\label{l3}
The operator field $\widetilde M$ defined by \eqref{lap3}  satisfies the relation
$$
\widetilde M^2  =  c_1 M^1 + \dots +  c_n M^n.
$$
\end{Lemma}

\begin{Remark}{\rm
Notice that $\widetilde M(u, c^0) = \frac{1}{2} \Id$. This allows us to rewrite the above relation in the form 
\begin{equation}
\label{eq:b8}
\widetilde M = \sqrt{c_1 M^1 + \dots +  c_n M^n},
\end{equation}
where the square root is a single-valued analytic matrix function which, by definition, is the inverse function to 
the map $M \mapsto M^2$ in the neighborhood of $\frac{1}{2}\Id$.   
}\end{Remark}

\begin{proof}
By definition,  $W(u, c)$ satisfies the following relations
\begin{equation}\label{rap1}
h^{jk}_i \pd{W}{u^j} \pd{W}{u^k} = c_i, \ \  \mbox{or shortly} \ \ h_i\left( \ddd W, \ddd W \right)  = c_i, \quad i = 1, \dots, n.    
\end{equation}
As $\widetilde M \in \mathfrak A$, the operator $\widetilde M^*$ is self-adjoint with respect to $h_i$. Hence, \eqref{rap1} can be written as
\begin{equation}\label{rap2}
h_i\bigl((\widetilde M^2)^* \alpha , \alpha\bigr) = h_i ( \widetilde M^* \alpha, \widetilde M^*\alpha )= h_i ( \tfrac{1}{2}\ddd W,  \tfrac{1}{2}\ddd  W)  =   \tfrac{1}{4}c_i, \quad i = 1, \dots, n.
\end{equation}
Using \eqref{eq:b1}, we obtain 
$$
c_i =   4\, h_i\bigl((\widetilde M^2)^* \alpha , \alpha\bigr) =  4\, h^{jk}_i (\widetilde M^2)_j^s g_s g_k = 2\,\delta^j_i (\widetilde M^2)_j^s g_s  = 2\,(\widetilde M^2)_i^s g_s, 
$$
which implies
\begin{equation}
\label{eq:b6}
2 (\widetilde M^2)^* \alpha = c_1 \ddd u^1 + \dots + c_n \ddd u^n.
\end{equation}

On the other hand, since $\widetilde M^2 \in \mathfrak A$ and $M^1,\dots, M^n$ form a basis in $\mathfrak A$,   we have
\begin{equation}
\label{eq:b7}
\widetilde M^2 = q_1 M^1 + \dots + q_n M^n
\end{equation}
for some smooth functions $q_i=q_i(u,c)$, and using \eqref{form1}  we obtain
\begin{equation}
\label{eq:b5}
2 (\widetilde M^2)^* \alpha =  
2  q_1 \bigl((M^1)^*\alpha\bigr) + \dots + 2 q_n \bigl((M^n)^* \alpha\bigr) =   
q_1 \ddd u^1 + \dots + q_n \ddd u^n.
\end{equation}

Now, from \eqref{eq:b6} and \eqref{eq:b5}, we see that $q_i = c_i$ and formula \eqref{eq:b7} coincides  with the statement of the Lemma. \end{proof}

The remaining part of the proof  basically follows from the differentiation of the relation $\ddd W  = 2\widetilde M^* \alpha$ with respect to the parameters $c_i$ at the point $c=c^0$. We use the fact that on  the right hand side,  only $\widetilde M$ depends on $c_i$, and this dependence is defined by \eqref{eq:b8}, i.e.,
$\widetilde M = \sqrt{c_1M^1+\dots + c_nM^n}$. Since $\widetilde M|_{c=c_0}=\frac{1}{2}\Id$ and $M^i$ pairwise commute, we get 
$$
\frac{\partial \widetilde M}{\partial c_i}|_{c=c_0} =   M^i,  \quad
\frac{\partial^2 \widetilde M}{\partial c_i\partial c_j}|_{c=c_0} = - 2  M^i M^j, \quad
\frac{\partial^3 \widetilde M}{\partial c_i\partial c_j \partial c_k}|_{c=c_0} = 12  M^i M^j M^k.      
$$
Using these relations and the formula \eqref{form1}, we obtain
$$
\ddd  \left( \frac{\partial^2  W}{\partial c_i\partial c_j}|_{c=c_0} \right) = \frac{\partial^2 }{\partial c_i\partial c_j}|_{c=c_0} \ddd W =  
\frac{\partial^2}{\partial c_i\partial c_j}|_{c=c_0} \Bigl( 2 \widetilde M^* \alpha \Bigr) = -4  {M^i}^* {M^j}^* \alpha = 
-2 {M^i}^* \ddd u^j.
$$
This shows that $\ddd u^j$ are common conservation laws  for all $M^i$. Similarly,
$$
\ddd \left(\frac{\partial^3 W}{\partial c_i \partial c_j \partial_k} \vert_{c = c_0}\right) = 24 (M^i)^* (M^j)^* (M^k)^* \ddd f = 12 (M^i)^* (M^j)^* \ddd u^k.
$$
Thus,  $\ddd u^j$ are also conservation laws for all pairwise products  $M^l M^j$. Lemma \ref{dfg3} implies that all $M^i$ are strong symmetries of each other. In particular, they are all Nijenhuis operators.

It remains to recall that $\alpha = \ddd W(u, c_0)$ is a closed form, which is a common conservation law for $M^i$ due to \eqref{form1}.  Since $p_i$ are momenta dual to the forms $\ddd u^i = 2(M^i)^*\alpha$,  we conclude that all the ingredients of \eqref{eq:b3} satisfy the assumptions of the direct procedure, which completes the proof.
\end{proof}

Notice that the symmetry algebra  $\Sym(M^1,\dots,M^n)$ can be reconstructed from the Killing tensors $K_1,\dots,K_n$ by a simple and purely algebraic procedure.  As before, we set $K_i\cdot K_j = \sum_s a_{ij}^s K_s$.
Following the initial procedure, we may choose a collection of constants $a_1, \dots, a_n\in\R$ for which the bilinear form $\bar b_{ij} = a^s_{ij} a_s$ is non-degenerate. 
\begin{Corollary}
Under  the assumptions of Theorem \ref{t6}, consider the operators $\bar M^i$  that are dual to $K_s$  with respect to the form $\bar b_{ij}$ in the sense of Theorem \ref{t2}, that is
$$
\bar M^i = \bar b^{is} K_s.
$$
They are related to $M_i$ constructed in the proof of Theorem \ref{t6} as follows:
$$
\bar M^i = (a_1 M^1+\dots + a_nM^n)^{-1} M^i.
$$
In particular, $\bar M^1, \dots, \bar M^n$ are all Nijenhuis operators, forming a basis of the same symmetry algebra $\Sym$, i.e., $\Sym(\bar M^1,\dots,\bar M^n)=\Sym(M^1,\dots, M^n)$.
\end{Corollary}
\begin{proof}
Consider the $1$-form $\beta = a_1\dd u^1 + \dots + a_n\dd u^n$, where $u^i$ denote the coordinates used in the proof of Theorem \ref{t6}.  Then we have
$$
\langle \bar {M^i}^*\beta; \partial_j \rangle = \langle \bar {M^i}^*\beta; K_j \partial_1 \rangle = \langle \beta; \bar b^{i \beta}K_\beta K_j \partial_1 \rangle = \langle \beta; \bar b^{i \beta} a_{\beta j}^s \partial_s \rangle = \bar b^{i \beta} a_{\beta j}^s \beta_s = \bar b^{i \beta} \bar b_{\beta j} = \delta^i_j
$$
This implies that $\bar {M^i}^* \beta = \ddd u^i$. At the same time,  by Lemma \ref{l2} we have  ${M^i}^*\alpha = \frac{1}{2}\dd u^i$ and therefore 
$$
2(a_1 {M^1}^* + \dots + a_n {M^n}^*)\alpha = \beta   \quad \mbox{or, equivalently,} \quad  2\alpha = (a_1 {M^1}^* + \dots + a_n {M^n}^*)^{-1} \beta.
$$
We write the formula
$$
{M^i}^* (a_1 {M^1}^* + \dots + a_n {M^n}^*)^{-1} \beta = {M^i}^* (2\alpha) = \ddd u^i = \bar {M^i}^* \beta.
$$
We rewrite this as
$$
\Big(\bar {M^i}^* - {M^i}^* (a_1 {M^1}^* + \dots + a_n {M^n}^*)^{-1}\Big) \beta = 0.
$$
As $M^i, \bar M^j$ lie pointwise in the same Frobenius algebra and $\beta$ is a generic covector, we conclude that  $\bar M^i = M^i (a_1 M^1 + \dots + a_n M^n)^{-1}$, as required.
\end{proof}


\vspace{1ex}
\noindent{\bf Data sharing  }  is not applicable to this article, as no datasets were generated or analysed during the current study. The authors declare {\bf  no conflicts of interest}.

\vspace{1ex}
 
\noindent{\bf Declaration on the use of generative AI.}  The authors obtained the main results without the use of generative AI, before publicly 
available tools were capable of providing substantial assistance with this work. During the final preparation of the manuscript, the authors used  ChatGPT (6.2 Astra)  to help check the mathematical arguments and improve the grammar. The authors carefully reviewed the resulting suggestions and take full responsibility for the final text.

\printbibliography

@article{LPG2024,
  author  = {Lorenzoni, P. and Perletti, S. and van Gemst, K.},
  title   = {The generalised hodograph method for non-diagonalisable integrable systems of hydrodynamic type},
  journal = {Nonlinearity},
  volume  = {39},
  number  = {5},
  pages   = {055005},
  year    = {2026},
  doi     = {10.1088/1361-6544/ae65fd},
  url     = {https://doi.org/10.1088/1361-6544/ae65fd}
}

@article {KM1980,
    AUTHOR = {Kalnins, E. G. and Miller,  Jr., W. },
     TITLE = {Killing tensors and variable separation for
              {H}amilton-{J}acobi and {H}elmholtz equations},
   JOURNAL = {SIAM J. Math. Anal.},
  FJOURNAL = {SIAM Journal on Mathematical Analysis},
    VOLUME = {11},
      YEAR = {1980},
    NUMBER = {6},
     PAGES = {1011--1026},
      ISSN = {0036-1410},
   MRCLASS = {53B20 (35C99 70H20)},
  MRNUMBER = {595827},
MRREVIEWER = {A.\ P.\ Stone},
       URL = {https://doi.org/10.1137/0511089},
}

@article{Lax,
    AUTHOR = {Lax, P. D.},
     TITLE = {Hyperbolic systems of conservation laws. {II}},
   JOURNAL = {Comm. Pure Appl. Math.},
  FJOURNAL = {Communications on Pure and Applied Mathematics},
    VOLUME = {10},
      YEAR = {1957},
     PAGES = {537--566},
      ISSN = {0010-3640,1097-0312},
   MRCLASS = {35.00},
  MRNUMBER = {93653},
MRREVIEWER = {M.\ A.\ Hyman},
       DOI = {10.1002/cpa.3160100406},
       URL = {https://doi.org/10.1002/cpa.3160100406},
}

@book {LeVe,
    AUTHOR = {LeVeque, R. J.},
     TITLE = {Finite volume methods for hyperbolic problems},
    SERIES = {Cambridge Texts in Applied Mathematics},
 PUBLISHER = {Cambridge University Press, Cambridge},
      YEAR = {2002},
     PAGES = {xx+558},
   MRCLASS = {65-01 (65M06 74S30 76M12)},
  MRNUMBER = {1925043},
MRREVIEWER = {Serge\ Piperno},
       DOI = {10.1017/CBO9780511791253},
       URL = {https://doi.org/10.1017/CBO9780511791253},
}

@article{Riemann,
  author       = {Riemann, B.},
  title        = {Über die Fortpflanzung ebener Luftwellen von endlicher Schwingungsweite},
  journal      = {Abhandlungen der K{\"o}niglichen Gesellschaft der Wissenschaften zu G{\"o}ttingen},
  volume       = {8},
  pages        = {43--66},
  year         = {1860},
  language     = {german}
}

@article {ManinHertling,
    AUTHOR = {Hertling, C. and Manin, Yu.},
     TITLE = {Weak {F}robenius manifolds},
   JOURNAL = {Internat. Math. Res. Notices},
  FJOURNAL = {International Mathematics Research Notices},
      YEAR = {1999},
    NUMBER = {6},
     PAGES = {277--286},
      ISSN = {1073-7928,1687-0247},
   MRCLASS = {53D45},
  MRNUMBER = {1680372},
MRREVIEWER = {Alexandre\ I.\ Kabanov},
       DOI = {10.1155/S1073792899000148},
       URL = {https://doi.org/10.1155/S1073792899000148},
}

@article {KoPa,
    AUTHOR = {Kokkinos, D. and Papakostas, T.},
     TITLE = {The study of the canonical forms of {K}illing tensor in vacuum
              with {$\Lambda$}},
   JOURNAL = {Gen. Relativity Gravitation},
  FJOURNAL = {General Relativity and Gravitation},
    VOLUME = {56},
      YEAR = {2024},
    NUMBER = {11},
     PAGES = {Paper No. 134, 40},
      ISSN = {0001-7701,1572-9532},
   MRCLASS = {83C20 (83C60)},
  MRNUMBER = {4819936},
MRREVIEWER = {Nikolaos\ Dimakis},
       DOI = {10.1007/s10714-024-03321-w},
       URL = {https://doi.org/10.1007/s10714-024-03321-w},
}

@book {Hertlingbook,
    AUTHOR = {Hertling, C.},
     TITLE = {Frobenius manifolds and moduli spaces for singularities},
    SERIES = {Cambridge Tracts in Mathematics},
    VOLUME = {151},
 PUBLISHER = {Cambridge University Press, Cambridge},
      YEAR = {2002},
     PAGES = {x+270},
      ISBN = {0-521-81296-8},
   MRCLASS = {32S40 (14B05 34M35 53D45)},
  MRNUMBER = {1924259},
MRREVIEWER = {Ezra\ Getzler},
       DOI = {10.1017/CBO9780511543104},
       URL = {https://doi.org/10.1017/CBO9780511543104},
}

@article {FeFo,
    AUTHOR = {Ferapontov, E. V. and Fordy, A. P.},
     TITLE = {Separable {H}amiltonians and integrable systems of
              hydrodynamic type},
   JOURNAL = {J. Geom. Phys.},
  FJOURNAL = {Journal of Geometry and Physics},
    VOLUME = {21},
      YEAR = {1997},
    NUMBER = {2},
     PAGES = {169--182},
      ISSN = {0393-0440,1879-1662},
   MRCLASS = {58F05 (58F07 70H05)},
  MRNUMBER = {1427864},
MRREVIEWER = {Efthymia\ Meletlidou},
       DOI = {10.1016/S0393-0440(96)00013-7},
       URL = {https://doi.org/10.1016/S0393-0440(96)00013-7},
}

@article{Eisenhart34,
    AUTHOR = {Eisenhart, L. P.},
     TITLE = {Separable systems of {S}tackel},
   JOURNAL = {Ann. of Math. (2)},
  FJOURNAL = {Annals of Mathematics. Second Series},
    VOLUME = {35},
      YEAR = {1934},
    NUMBER = {2},
     PAGES = {284--305},
      ISSN = {0003-486X},
   MRCLASS = {DML},
  MRNUMBER = {1503163},
       URL = {https://doi.org/10.2307/1968433},
}

@book {Serre2,
    AUTHOR = {Serre, D.},
     TITLE = {Systems of conservation laws. 1},
      NOTE = {Hyperbolicity, entropies, shock waves,
              Translated from the 1996 French original by I. N. Sneddon},
 PUBLISHER = {Cambridge University Press, Cambridge},
      YEAR = {1999},
     PAGES = {xxii+263},
      ISBN = {0-521-58233-4},
   MRCLASS = {35L65 (35B35 35D05 35L67)},
  MRNUMBER = {1707279},
       DOI = {10.1017/CBO9780511612374},
       URL = {https://doi.org/10.1017/CBO9780511612374},
}

@incollection {Serre1,
    AUTHOR = {Serre, D.},
     TITLE = {Systems of conservation laws: a challenge for the {XXI}st
              century},
 BOOKTITLE = {Mathematics unlimited---2001 and beyond},
     PAGES = {1061--1080},
 PUBLISHER = {Springer, Berlin},
      YEAR = {2001},
      ISBN = {3-540-66913-2},
   MRCLASS = {35L65},
  MRNUMBER = {1852203},
}

@book {Serre3,
    AUTHOR = {Serre, D.},
     TITLE = {Systems of conservation laws. 2},
      NOTE = {Geometric structures, oscillations, and initial-boundary value
              problems,
              Translated from the 1996 French original by I. N. Sneddon},
 PUBLISHER = {Cambridge University Press, Cambridge},
      YEAR = {2000},
     PAGES = {xii+269},
      ISBN = {0-521-63330-3},
   MRCLASS = {35L65 (35L67)},
  MRNUMBER = {1775057},
}

@article{st,
  author        = {Bolsinov, A. V. and Konyaev, A. Yu. and Matveev, V. S.},
  title         = {St{\"a}ckel problem for non-diagonal Killing tensors: Yano-Patterson lifts, algebra of strong symmetries and quadratic in momenta integrals},
  journal       = {arXiv:2512.17609},
  year          = {2025},
 % eprint        = {2512.17609},
  %archivePrefix = {arXiv},
  %primaryClass  = {nlin.SI},
  doi           = {10.48550/arXiv.2512.17609},
  note ={accepted  in J. Nonl. Sci.}
 % url           = {https://arxiv.org/abs/2512.17609}
}

@article {staeckel,
AUTHOR ={St\"ackel, P.},
TITLE = {Die Integration der Hamilton-Jacobischen Differentialgleichung mittelst Separation der Variablen},
JOURNAL = {Habilitationsschrift, Universit\"at Halle}, 
Year= {1891}
}

@book{Dubrovin1,
    AUTHOR = {Donagi, R. and Dubrovin, B. and Frenkel, E. and Previato, E.},
     TITLE = {Integrable systems and quantum groups},
    SERIES = {Lecture Notes in Mathematics},
    VOLUME = {1620},
    EDITOR = {Francaviglia, M. and Greco, S.},
      NOTE = {Lectures given at the First 1993 C.I.M.E. Session held in
              Montecatini Terme, June 14--22, 1993,
              Fondazione CIME/CIME Foundation Subseries},
 PUBLISHER = {Springer-Verlag, Berlin; Centro Internazionale Matematico
              Estivo (C.I.M.E.), Florence},
      YEAR = {1996},
     PAGES = {viii+488},
      ISBN = {3-540-60542-8},
   MRCLASS = {58-06 (00B25 14-06 17-06)},
  MRNUMBER = {1397272},
       DOI = {10.1007/BFb0094791},
       URL = {https://doi.org/10.1007/BFb0094791},
}

@article {Dubrovin2,
    AUTHOR = {Dubrovin, B. and Zhang, Y.},
     TITLE = {Extended affine {W}eyl groups and {F}robenius manifolds},
   JOURNAL = {Compositio Math.},
  FJOURNAL = {Compositio Mathematica},
    VOLUME = {111},
      YEAR = {1998},
    NUMBER = {2},
     PAGES = {167--219},
      ISSN = {0010-437X,1570-5846},
   MRCLASS = {32M10 (14B07 17B20)},
  MRNUMBER = {1606165},
MRREVIEWER = {Lowell\ Abrams},
       DOI = {10.1023/A:1000258122329},
       URL = {https://doi.org/10.1023/A:1000258122329},
}

@article {BlMa,
    AUTHOR = {B{\l}aszak, M. and Marciniak, K.},
     TITLE = {From {S}t\"{a}ckel systems to integrable hierarchies of
              {PDE}'s: {B}enenti class of separation relations},
   JOURNAL = {J. Math. Phys.},
  FJOURNAL = {Journal of Mathematical Physics},
    VOLUME = {47},
      YEAR = {2006},
    NUMBER = {3},
     PAGES = {032904, 26},
      ISSN = {0022-2488,1089-7658},
   MRCLASS = {37J35 (35Q51 37K10)},
  MRNUMBER = {2219792},
MRREVIEWER = {Giovanni\ Rastelli},

       URL = {https://doi.org/10.1063/1.2176908},
}

@article {agafonov,
    AUTHOR = {Agafonov, S. I. and Matveev, V. S.},
     TITLE = {Integrable geodesic flows with simultaneously diagonalisable
              quadratic integrals},
   JOURNAL = {Arnold Math. J.},
  FJOURNAL = {Arnold Mathematical Journal},
    VOLUME = {11},
      YEAR = {2025},
    NUMBER = {4},
     PAGES = {1--9},
      ISSN = {2199-6792,2199-6806},
   MRCLASS = {37J35 (37D40 70H06)},
  MRNUMBER = {5032074},
MRREVIEWER = {Brayan\ Ferreira},
       DOI = {10.56994/armj.011.004.001},
       URL = {https://doi.org/10.56994/armj.011.004.001},
}

@article {nijenhuis,
    AUTHOR = {Nijenhuis, A.},
     TITLE = {{$X_{n-1}$}-forming sets of eigenvectors},
      NOTE = {Nederl. Akad. Wetensch. Proc. Ser. A {\bf 54}},
   JOURNAL = {Indag. Math.},
  FJOURNAL = {},
    VOLUME = {13},
      YEAR = {1951},
     PAGES = {200--212},
   MRCLASS = {53.0X},
  MRNUMBER = {43540},
MRREVIEWER = {J.\ A.\ Schouten},
}

@article {nij3,
    AUTHOR = {Bolsinov, A. V. and Konyaev, A. Yu. and Matveev,
              V. S.},
     TITLE = {Nijenhuis geometry {III}: gl-regular {N}ijenhuis operators},
   JOURNAL = {Rev. Mat. Iberoam.},
  FJOURNAL = {Revista Matem\'{a}tica Iberoamericana},
    VOLUME = {40},
      YEAR = {2024},
    NUMBER = {1},
     PAGES = {155--188},
      ISSN = {0213-2230,2235-0616},
   MRCLASS = {53C15 (35A30 35B06 35F20 35N10)},
  MRNUMBER = {4705683},
       URL = {https://doi.org/10.4171/rmi/1416}
}

@article {l,
    AUTHOR = {Arsie, A. and Buryak, A. and Lorenzoni, P.
              and Rossi, P.},
     TITLE = {Riemannian {F}-manifolds, bi-flat {F}-manifolds, and flat
              pencils of metrics},
   JOURNAL = {Int. Math. Res. Not. IMRN},
  FJOURNAL = {International Mathematics Research Notices. IMRN},
      YEAR = {2022},
       VOLUME = {2022},
    NUMBER = {21},
     PAGES = {16730--16778},
      ISSN = {1073-7928,1687-0247},
   MRCLASS = {53D45},
  MRNUMBER = {4504906},
MRREVIEWER = {Robert\ Brouzet},
       URL = {https://doi.org/10.1093/imrn/rnab203},
}

@article{Arsie1,
author = {Arsie, A. and Lorenzoni, P.},
title = {Biflat F-structures as differential bicomplexes and Gauss–Manin connections},
journal = {Bulletin of the London Mathematical Society},
volume = {57},
number = {3},
pages = {786-808},
doi = {10.1112/blms.70000},
%url = {https://londmathsoc.onlinelibrary.wiley.com/doi/abs/10.1112/blms.70000},
%eprint = {https://londmathsoc.onlinelibrary.wiley.com/doi/pdf/10.1112/blms.70000},
year = {2025}
}

@article{Arsie2,
author = {Arsie, A. and Lorenzoni, P.},
title = {Cyclic F-manifolds, distinguished connections and integrability},
journal = {preprint arXiv:2605.25277v2},
year = {2026},
number = {},
pages = {},
doi = {10.48550/arXiv.2605.25277
},
url = 
{https://doi.org/10.48550/arXiv.2605.25277}
}

@article{BKM2026,
 author  = {Bolsinov, A. V. and Konyaev, A. Yu. and Matveev, V. S.},
  title   = {Existence of Riemannian invariants for integrable systems of hydrodynamic type},
  journal = {Physica D: Nonlinear Phenomena},
  volume  = {495},
  pages   = {135304},
  year    = {2026},
  doi     = {10.1016/j.physd.2026.135304},
  url     = {https://doi.org/10.1016/j.physd.2026.135304}
}

@article {tsarev,
    AUTHOR = {Tsar\"ev, S. P.},
     TITLE = {The geometry of {H}amiltonian systems of hydrodynamic type.
              {T}he generalized hodograph method},
   JOURNAL = {Izv. Akad. Nauk SSSR Ser. Mat.},
  FJOURNAL = {Izvestiya Akademii Nauk SSSR. Seriya Matematicheskaya},
    VOLUME = {54},
      YEAR = {1990},
    NUMBER = {5},
     PAGES = {1048--1068},
      ISSN = {0373-2436},
   MRCLASS = {58F07 (35Q35 58F05 76B99 76M99)},
  MRNUMBER = {1086085},
MRREVIEWER = {Galina\ G.\ Okuneva},
       DOI = {10.1070/IM1991v037n02ABEH002069},
       URL = {https://doi.org/10.1070/IM1991v037n02ABEH002069},
}

@article{Jacobi,
  author  = {Jacobi, C. G. J.},
  title   = {Note von der geod{\"a}tischen Linie auf einem Ellipsoid und den verschiedenen Anwendungen einer merkw{\"u}rdigen analytischen Substitution},
  journal = {Journal f{\"u}r die reine und angewandte Mathematik},
  volume  = {19},
  pages   = {309--313},
  year    = {1839},
  doi     = {10.1515/crll.1839.19.309}
}

@article {nij4,
    AUTHOR = {Bolsinov, A. V. and Konyaev, A.
              Yu. and Matveev,
              V. S.},
     TITLE = {Nijenhuis geometry {IV}: conservation laws, symmetries and
              integration of certain non-diagonalisable systems of
              hydrodynamic type in quadratures},
   JOURNAL = {Nonlinearity},
  FJOURNAL = {Nonlinearity},
    VOLUME = {37},
      YEAR = {2024},
    NUMBER = {10},
     PAGES = {Paper No. 105003, 27},
      ISSN = {0951-7715,1361-6544},
   MRCLASS = {53B30 (37K06 47N20 53A20 53B10 58J70)},
  MRNUMBER = {4792045},
}
\end{document}